\documentclass{article}

\usepackage[margin=2.5cm,includefoot,footskip=30pt]{geometry}
\usepackage{amsmath}
\usepackage{amsthm, cite}
\usepackage{amssymb}
\usepackage[ansinew]{inputenc}
\usepackage{url}
\usepackage{enumerate}

\theoremstyle{plain}
    \newtheorem{lemma}{Lemma}[section]
    \newtheorem{theorem}[lemma]{Theorem}
    \newtheorem{cor}[lemma]{Corollary}
    \newtheorem{prop}[lemma]{Proposition}
\theoremstyle{definition}
    \newtheorem{que}[lemma]{Problem}
    \newtheorem{defi}[lemma]{Definition}
    \newtheorem{example}[lemma]{Example}
\theoremstyle{remark}
    \newtheorem{rem}[lemma]{Remark}

\newcommand{\frn}{\mathfrak{n}}
\newcommand{\cfn}{\mathrel{\sim_{\frn}}}

\newcommand\gj{\mathcal{J}}
\newcommand\gd{\mathcal{D}}
\newcommand\gl{\mathcal{L}}
\newcommand\gr{\mathcal{R}}
\newcommand\gh{\mathcal{H}}

\newcommand{\greenJ}{\mathrel{\gj}}
\newcommand{\greenD}{\mathrel{\gd}}
\newcommand{\greenL}{\mathrel{\gl}}
\newcommand{\greenR}{\mathrel{\gr}}
\newcommand{\greenH}{\mathrel{\gh}}

\DeclareMathOperator{\PAut}{PAut}
\DeclareMathOperator{\Epi}{Epi}

\DeclareMathOperator{\Inn}{Inn}
\DeclareMathOperator{\End}{End}

\newcommand{\inv}{^{-1}}                
\newcommand{\Z}{\mathbb{Z}}
\DeclareMathOperator\dom{dom}
\DeclareMathOperator\ima{im}
\DeclareMathOperator\id{id}

\title{The Inverse Monoid\\ of Partial Inner Automorphisms of a Semigroup}

\author{Jo\~ao Ara\'{u}jo\footnote{Center for Mathematics and Applications (NOVA Math) \& Department of Mathematics, NOVA School of Science and Technology, NOVA University of Lisbon, 2829–516 Caparica, Portugal;
\texttt{jj.araujo@fct.unl.pt}},
Wolfram Bentz\footnote{Center for Mathematics and Applications (NOVA Math), NOVA School of Science and Technology, NOVA University of Lisbon, 2829–516 Caparica, Portugal; and
Universidade Aberta, R. Escola Polit\'ecnica, 147, 1269--001 Lisboa, Portugal;
\texttt{Wolfram.Bentz@uab.pt}},
Michael Kinyon\footnote{Department of Mathematics, University of Denver, Denver, CO 80208, USA;
\texttt{michael.kinyon@du.edu}},
Janusz Konieczny\footnote{Department of Mathematics, University of Mary Washington, Fredericksburg, VA 22401, USA;
\texttt{jkoniecz@umw.edu}},\\
Ant\'onio Malheiro\footnote{Center for Mathematics and Applications (NOVA Math) \& Department of Mathematics, NOVA School of Science and Technology, NOVA University of Lisbon, 2829–516 Caparica, Portugal;
\texttt{ajm@fct.unl.pt}},
and
Valentin Mercier\footnote{Center for Mathematics and Applications (NOVA Math), NOVA School of Science and Technology, NOVA University of Lisbon, 2829–516 Caparica, Portugal;
\texttt{valen.mercier@gmail.com}}}
\date{}
\begin{document}

\maketitle

\begin{abstract}
We introduce the inverse monoid of inner partial automorphisms of a semigroup --- a tool that associates to every semigroup an inverse semigroup. When the semigroup is a group, this inverse semigroup is isomorphic to the group of inner automorphisms with a zero adjoined.

We then describe this structure for completely simple semigroups, the full transformation monoid, and the endomorphism monoid of a finite $G$-set, when $G$ is a finite abelian group. The paper ends with some open problems.

\vskip 2mm

\noindent\emph{$2020$ Mathematics Subject Classification\/}. {20M10, 20M15, 20M20.}

\end{abstract}

\section{Introduction}

A \emph{partial automorphism} of an algebra $A$ (in the sense of universal algebra) is an isomorphism between subalgebras of $A$. The set $\PAut(A)$ generated by  all partial automorphisms of $A$ is an inverse monoid under composition of partial maps. The study of partial automorphisms of algebraic structures has a long history. O. I. Domanov gave an order theoretic characterization of those inverse monoids with $0$ which occur as partial automorphism monoids of universal algebras \cite{Domanov}. G. B. Preston posed the problem of characterizing algebras by their partial automorphism monoids \cite{Preston}. The partial automorphism monoid of semigroups was studied by S. M. Goberstein \cite{Goberstein1,Goberstein2} (see also \cite{FitzGerald1}). V. D. Derech classified those finite semigroups $S$ for which $\PAut(S)$ has all congruences permutable \cite{Derech} (also \cite{Derech2,Derech1}). More recently, P. J. Cameron proved that for a finite abelian group $A$, the underlying set of $\PAut(A)$ has the same cardinality as that of $\End(A)$, the monoid of all endomorphisms of $A$ \cite{pjc}.

Conjugacy in groups induces an intrinsic notion of automorphism. Two elements $a, b$ of a group $G$ are \emph{conjugate} if there exists $g\in G$ such that $b = g^{-1} a g$. For each $g\in G$, the transformation $G\to G; a\mapsto g^{-1} a g$ is an automorphism of $G$, known as an \emph{inner} automorphism.

There are many notions of conjugacy in semigroups \cite{ABKKMM,ArKiKoMaTA,AKM14}, and a growing body of research dedicated to them \cite{GarretaGray2021,Zhang,Zhang1992,Otto1984,NarendranOtto1985,NarendranOttoWinklmann1984,BorralhoKinyon2020,LiuWang2022,Jack2023,Jack2024}.  A natural one in inverse semigroups is the following: two elements $a,b$ of an inverse semigroup $S$ are \emph{i-conjugate} if there exists $g\in S^1$ such that $b = g\inv ag$ and $a = gbg\inv$ \cite{ArKiKo_Inv}. (As usual, $S^1 = S$ if $S$ is a monoid; otherwise $S^1 = S\cup \{1\}$ where $1$ is an adjoined identity element.) This type of conjugacy leads to \emph{partial inner} automorphisms of inverse semigroups: For each $g\in S$, the partial bijection $gSg\inv \to g\inv Sg; x\mapsto g\inv xg$ is an isomorphism of inverse subsemigroups. The set of all partial inner automorphisms of $S$ forms an inverse monoid. Partial inner automorphisms of inverse semigroups were first studied by G. B. Preston \cite{Preston2} and are related to the universal algebra center of inverse semigroups \cite{KS19}.

In 2018, the fourth author \cite{Ko18} defined a conjugacy $\cfn$ on any semigroup $S$ as follows: for $a,b \in S$, 
\begin{equation}\label{e1dcon2}
	a\cfn b\,\iff\, \exists_{g,h\in S^1}\ (\,ag=gb\,,\,\, bh=ha\,,\,\, hag=b \text{ and } gbh=a\,)\,.
\end{equation}
The relation $\cfn$, called \emph{natural conjugacy}, is an equivalence relation on any semigroup $S$. It turns out that in inverse semigroups, natural conjugacy coincides with i-conjugacy \cite{Ko18,ABKKMM}.

The goal of this paper is to show that natural conjugacy leads to a notion of partial inner automorphisms for arbitrary semigroups. These will form an inverse monoid. In the case where the underlying semigroup is itself inverse, this monoid will coincide with the monoid described above. In the group case, the partial inner automorphism monoid coincides with the usual inner automorphism group with a zero adjoined. Therefore, for every semigroup $S$, natural conjugacy $\cfn$ induces an inverse semigroup related in a particular way to $S$. It is natural to ask whether similar correspondences can be obtained from the other known conjugacy notions. The results in this paper appear to depend in a very delicate way on the specific features of $\cfn$.

In \S\ref{Sec:prelim}, we recall some results about the natural conjugacy relation. In \S\ref{Sec:mapping}, we introduce partial inner automorphisms and prove the main results about them. The rest of the paper is devoted to finding the partial inner automorphism monoids of various types of semigroups. In \S\ref{Sec:CS}, we compute the partial inner automorphisms of completely simple semigroups. In \S\ref{Sec:mappingT(X)}, we find the partial inner automorphisms of the full transformation monoid $T(X)$. In \S\ref{Sec:mappingG-set}, we find the partial inner automorphisms of the endomorphism monoid of a $G$-set for $G$ abelian. Finally, in \S\ref{Sec:questions}, we suggest some problems. 

\section{Preliminaries}
\label{Sec:prelim}

In this section we recall some general results about natural conjugacy $\cfn$ (or just $\frn$-conjugacy, for short) and its interactions with Green's relations. All these results are proved in (\!\!\cite{ABKKMM}, {\S}2). 

For a semigroup $S$, $a,b\in S$ and $g,h\in S^1$, consider the following equations:
\begin{center}
  \begin{tabular}{rccrc}
  (i)   & $ag = gb$       & \qquad & (ii)   & $bh= ha$ \\
  (iii) & $hag = b$       & \qquad & (iv)   & $gbh = a$ \\
  (v)   & $hg\cdot b = b$ & \qquad & (vi)   & $gh\cdot a = a$ \\
  (vii) & $b\cdot hg = b$ & \qquad & (viii) & $a\cdot gh = a$.
\end{tabular}
\end{center}
Our definition of $\cfn$ is based on (i), (ii), (iii) and (iv).
We now give some characterizations which will be useful later.

\begin{prop}\label{Prp:alternatives}
Let $S$ be a semigroup, and let $a,b\in S$ and $g,h\in S^1$. Each of the following sets of equations implies all of \textup{(i)--(viii)}, and thus $a\cfn b$.
\begin{center}
\begin{tabular}{rlcrl}
\textup{(1)} & $\{$\textup{(i),(iii),(iv)}$\}$ & \qquad &
\textup{(2)} & $\{$\textup{(ii),(iii),(iv)}$\}$ \\
\textup{(3)} & $\{$\textup{(i),(iii),(viii)}$\}$ & \qquad &
\textup{(4)} & $\{$\textup{(ii),(iv),(vii)}$\}$ \\
\textup{(5)} & $\{$\textup{(i),(iv)(v)}$\}$ & \qquad &
\textup{(6)} & $\{$\textup{(ii),(iii),(vi)}$\}$ \\
\textup{(7)} & $\{$\textup{(i),(v),(viii)}$\}$ & \qquad &
\textup{(8)} & $\{$\textup{(ii),(vi),(vii)}$\}$ \\
\textup{(9)} & $\{$\textup{(iii),(iv),(v)}$\}$ & \qquad &
\textup{(10)} & $\{$\textup{(iii),(iv),(vi)}$\}$ \\
\textup{(11)} & $\{$\textup{(iii),(iv),(vii)}$\}$ & \qquad &
\textup{(12)} & $\{$\textup{(iii),(iv),(viii)}$\}$ \\
\textup{(13)} & $\{$\textup{(iii),(vi),(viii)}$\}$ & \qquad &
\textup{(14)} & $\{$\textup{(iv),(v),(vii)}$\}$ \\
\textup{(15)} & $\{$\textup{(i),(ii),(v),(vii)}$\}$ & \qquad &
\textup{(16)} & $\{$\textup{(i),(ii),(vi),(viii)}$\}$ \\
\end{tabular}
\end{center}
\end{prop}

For a semigroup $S$, if $a,b\in S$ satisfy $a\cfn b$, then there exist \emph{conjugators} $g,h\in S^1$ satisfying all of the conditions (i)--(viii). We will often use (i)--(viii) without explicit reference.

For $a\in S$ we write $[a]_{\frn}=\{b\in S:b\cfn a\}$ for the $\frn$-conjugacy class of $a$.

\begin{rem}\label{Rem:n_conj_class_01}
Note that in any semigroup with a zero, $[0]_{\frn} =\{0\}$, and in any monoid $M$ with identity $1$, $[1]_{\frn}=\{gh\in M\,:\,hg=1\}$.
\end{rem}

\begin{prop}\label{Prp:subsemi}
Let $S$ be a semigroup, let $g,h\in S^1$, and let
\[
K_{g,h} = \{ (a,b)\in S\times S\,:\, a\cfn b\text{ with }g,h\text{ as conjugators }\}\,.
\]
Then $K_{g,h}$ is a subsemigroup of $S\times S$.
\end{prop}

\begin{cor}\label{Cor:powers}
Let $S$ be a semigroup and assume $a,b\in S$ satisfy $a\cfn b$ with conjugators $g,h\in S^1$. Then for all positive integers $k$, $a^k\cfn b^k$ with conjugators $g,h$.
\end{cor}

If $S$ is a semigroup and $a,b\in S$, we say that $a\greenL b$ if $S^1a = S^1b$, $a\greenR b$ if $aS^1 = bS^1$, and $a\greenJ b$ if $S^1aS^1 = S^1bS^1$. We define $\gh = \gl\cap \gr$, and $\gd$ as the join of $\gl$ and $\gr$, that is, the smallest equivalence relation on $S$ containing both $\gl$ and $\gr$. These five equivalence relations are known as \emph{Green's relations} \cite[p.~45]{Howie}, one of the most important tools in studying semigroups. The relations $\gl$ and $\gr$ commute \cite[Prop.~2.1.3]{Howie}; consequently, $\gd = \gl\circ\gr = \gr\circ\gl$. We have $\gd\subseteq \gj$, and in epigroups, such as finite or periodic semigroups, $\gd = \gj$ \cite[Prop.~2.1.4]{Howie}.

\begin{prop}\label{Prp:D}
	Let $S$ be a semigroup. Then:
	\begin{enumerate}
		\item\quad $\cfn\ \subseteq\ \greenD$;
		\item\quad for all $e,a\in S$, if $e$ is an idempotent and $e\cfn a$, then $a$ is also an idempotent;
		\item\quad Restricted to idempotents, $\cfn$ and $\greenD$ coincide.
	\end{enumerate}
\end{prop}


    

A pair $g,h$ of elements of a semigroup $S$ are said to be \emph{mutually inverse} if $ghg=g$ and $hgh=h$.

\begin{theorem}\label{Thm:D-idem}
  Let $S$ be a semigroup and let $e,f\in S$ be idempotents. Then $e\cfn f$ if and only if $e\greenD f$. Furthermore, if $e\cfn f$, then there exist mutually inverse conjugators $g,h$ in the same $\greenD$-class as $e,f$.
\end{theorem}

\section{Conjugacy $\cfn$ and partial inner automorphisms}
\label{Sec:mapping}
\setcounter{equation}{0}

Let $S$ be a semigroup, fix $g,h\in S^1$, and define
\[
D_{g,h} = \{ a\in S \mid gh\cdot a = a\cdot gh = a \}\,.
\]
Note that for all $a,b\in S$, $a\cfn b$ with conjugators $g$ and $h$ if and only if $a\in D_{g,h}$ and $b = hag$ (see Proposition~\ref{Prp:alternatives}).

Let $\preceq$ be a preorder (a reflexive, transitive binary relation) on a set $A$. We say that a subset $B$ of $A$ is \emph{downward directed} in $\preceq$ if for all $a\in A$ and $b\in B$, $a\preceq b$ implies $a\in B$.

Let $S$ be a semigroup. Then the relation $\preceq_{\greenH}$ on $S$ defined by $a\preceq_{\greenH}b$ if $sb=a=bt$
for some $s,t\in S^1$ is a preorder on $S$. Note that if $a\preceq_{\greenH}b$ and $b\preceq_{\greenH}a$, then $a\greenH b$.

The \emph{natural partial order} $\leq$ on a semigroup $S$ \cite{Mitsch} is defined by
\[
a\leq b\,\iff\, \exists s,t\in S^1\ (\ sa = a = at \text{ and } sb = a = bt \ )\,.
\]
The natural partial order refines the $\greenH$-preorder.

\begin{lemma}\label{Lem:domain_facts}
Let $S$ be a semigroup and let $g,h\in S^1$. Then:
\begin{itemize}
  \item[\textup{(1)}] $D_{g,h}$ is a subsemigroup of $S$;
  \item[\textup{(2)}] $D_{g,h}$ is downward directed in the $\greenH$-preorder $\preceq_{\greenH}$;
  \item[\textup{(3)}] $D_{g,h}$ is downward directed in the natural partial order $\leq$;
  \item[\textup{(4)}] if $a\in D_{g,h}$, then $H_a\subseteq D_{g,h}$, where $H_a$ denotes the
  $\greenH$-class of $a$ in $S$.
\end{itemize}
\end{lemma}
\begin{proof}
(1) is clear. For (2), assume $a\in D_{g,h}$ and $c\preceq_{\greenH} a$. Then there exist
$s,t\in S^1$ such that $sa = c = at$. We have $c\cdot gh = s\underbrace{a\cdot gh} = sa = c$
and $gh\cdot c = \underbrace{gh\cdot a}t = a t = c$, and so $c\in D_{g,h}$, as claimed.
Now (3) follows from (2) since $\leq$ refines $\preceq_{\greenH}$. Finally, (4) also follows from (2).
\end{proof}

Now we define a mapping by
\[
\phi_{g,h} : D_{g,h}\to S ; a\mapsto hag\,.
\]

\begin{theorem}\label{Thm:phi_iso}
The mapping  $\phi_{g,h}$ is a partial automorphism of $S$, specifically, it is an isomorphism
 from $D_{g,h}$ to $D_{h,g}$.
\end{theorem}
\begin{proof}
  For $a\in D_{g,h}$, set $b = a\phi_{g,h} = hag$. By Proposition \ref{Prp:alternatives}, $a\cfn b$ with $g,h$ as conjugators. Thus we also have $hg\cdot b = b\cdot hg = b$, that is, $b\in D_{h,g}$. In addition, $gbh = a$, that is, $b\phi_{h,g} = a$. Since
  $a\phi_{g,h}\phi_{h,g} = ghagh = a$ and $b\phi_{h,g}\phi_{g,h} = b$, we have that $\phi_{g,h}$ is a bijection from $D_{g,h}$ to $D_{h,g}$.

  Finally we show that $\phi_{g,h}$ is a homomorphism. Let $a_1,a_2\in D_{g,h}$ be given and set $b_i = ha_i g$ for $i=1,2$. Since $a_i\cfn b_i$, we have $(a_1 a_2)\phi_{g,h} = h a_1 \underbrace{a_2 g} =
  \underbrace{h a_1 g} b_2 = b_1 b_2$, which establishes the claim.
\end{proof}

\begin{theorem}\label{Thm:H_bijections}
    The bijection $\phi_{g,h} : D_{g,h}\to D_{h,g}$ restricts to bijections between $\greenH$-classes, that is, for $a\in D_{g,h}$ and $b = a\phi_{g,h}$, the restriction of $\phi_{g,h}$ to $H_a$ is a bijection onto $H_b$. Further, if $H_a$ is a group $\greenH$-class then $\phi_{g,h}$ is a group isomorphism.
\end{theorem}
\begin{proof}
    Fix $c\in H_a$ and let $d = c\phi_{g,h} = hcg$. There exist $s_1,s_2,t_1,t_2\in S^1$ such that $s_1 a = c$, $s_2 c = a$, $a t_1 = c$, $c t_2 = a$. Set $\bar{s}_i = h s_i g$ and $\bar{t}_i = h t_i g$ for $i=1,2$. Then
    \begin{align*}
      \bar{s}_1 b &= h s_1 \underbrace{g b} = h \underbrace{s_1 a} g = h c g = d\,, \\
      \bar{s}_2 d &= h s_2 \underbrace{g h c} g = h s_2 c g = h a g = b\,, \\
      b \bar{t}_1 &= \underbrace{b h} t_1 g = h \underbrace{a t_1} g = h c g = d\quad\text{and} \\
      d \bar{t}_2 &= h \underbrace{c g h} t_2 g = h \underbrace{c t_2} g = h a g = b\,.
    \end{align*}
    This proves $d\greenH b$. Thus $(H_a)\phi_{g,h}\subseteq H_b$ and by symmetry, $(H_b)\phi_{h,g}\subseteq H_a$. Finally $H_b = (H_b)\phi_{h,g}\phi_{g,h}\subseteq (H_a)\phi_{g,h}\subseteq H_b$, so that $\phi_{g,h}$ is a bijection of $H_a$ onto $H_b$. The remaining assertion follows from Theorem \ref{Thm:phi_iso}.
\end{proof}

\begin{rem}
It is a basic result in semigroup theory that any two group $\greenH$-classes in the same $\greenD$-class of a semigroup are isomorphic \cite[Prop.~2.3.6]{Howie}. We have actually reproved this; it follows from Theorem \ref{Thm:D-idem} and Theorem \ref{Thm:H_bijections}. Our proofs are certainly more involved but better highlight the role of $\frn$-conjugacy.
\end{rem}

\begin{cor}\label{Cor:Hcommute}
 $\greenH\circ \cfn\ =\ \cfn\circ \greenH$.
\end{cor}
\begin{proof}
  Say $c\greenH a\cfn b$ and let $g,h\in S^1$ be conjugators for $a,b$. Set $d = (c)\phi_{g,h}$. By Theorem \ref{Thm:H_bijections}, we have $b\greenH d\cfn c$. The other inclusion is similarly proved.
\end{proof}

Now we consider the composition of partial automorphisms.

\begin{prop}
  For $g_i,h_i\in S^1$, $i=1,2$, we have
    \begin{equation}\label{Eq:aut_comp}
    \phi_{g_1,h_1}\phi_{g_2,h_2} \subseteq \phi_{g_1 g_2,h_2 h_1}\,.
    \end{equation}
\end{prop}
\begin{proof}
The domain of $\phi_{g_1,h_1}\phi_{g_2,h_2}$ is
\[
C = \{a\in D_{g_1,h_1}\mid h_1 a g_1 \in D_{g_2,h_2}\}\,.
\]
If $a\in C$, then
\begin{align*}
  g_1 g_2 h_2 h_1\cdot a &= g_1 \underbrace{g_2 h_2 h_1 a g_1} h_1
= g_1 h_1 a g_1 h_1 = a \\
\intertext{and}
  a\cdot g_1 g_2 h_2 h_1 &= g_1 \underbrace{h_1 a g_1 g_2 h_2} h_1
= g_1 h_1 a g_1 h_1 = a\,.
\end{align*}
Thus $a\in D_{g_1 g_2,h_2 h_1}$. Clearly
$a \phi_{g_1,h_1}\phi_{g_2,h_2} = a \phi_{g_1 g_2,h_2 h_1}$ for
$a\in C$.
\end{proof}

\begin{example}\label{Exm:proper}
In general, the inclusion \eqref{Eq:aut_comp} can be proper. For instance, in the group $\Z_2$ written additively, the map $\phi_{0,1}$ is the empty map and thus so is $\phi_{0,1}\phi_{0,1}$. However,
$\phi_{0+0,1+1} = \phi_{0,0}$ is the identity map.
\end{example}

\begin{defi}
Let $\Inn(S)$ denote the inverse monoid of partial automorphisms generated by the $\phi_{g,h}$'s. We call $\Inn(S)$ the \emph{partial inner automorphism monoid} of $S$.
\end{defi}

The following remarks are intended to show that the partial inner automorphism monoid of a semigroup is a natural generalization of the inner automorphism group of a group. 

\begin{rem}\label{Rem:cancellative}
	(1) Let $S$ be a cancellative monoid with unit group $U(S)$. For $a,g,h\in S$, $a\in D_{g,h}$ if and only if $gha = a = agh$ if and only if $gh = 1$ (by cancellation), in which case, $h = hgh$ so $hg=1$ as well. Thus if $D_{g,h}\neq \varnothing$, then $g\in U(S)$, $h = g\inv$, and $D_{g,g\inv} = S$. Thus for $g\in U(S)$, $\phi_{g,g\inv}$ is the usual conjugation by $g$, that is, $(x)\phi_{g,g\inv} = g\inv xg$ for all $x\in S$. Since $\phi_{g,g\inv} \phi_{k,k\inv} = \phi_{gk,(gh)\inv}$ for $g,k\in U(S)$, we see that $\Inn(S)$ is the union of the group $\{\phi_{g,g\inv}\,:\,g\in U(S)\}$ and the empty mapping.
	\medskip
	
	\noindent (2) In the special case where $S = U(S)$ is a group, our $\Inn(S)$ is a zero group, the union of the usual inner automorphism group of $S$ and the empty mapping.
	\medskip
	
	\noindent (3) On the other hand, let $S$ be the free monoid on a fixed nonempty set. Then $S$ is cancellative, but $U(S)$ is the trivial group.
	Thus $\Inn(S)$ consists only of the identity mapping and the empty mapping.
\end{rem}

\begin{rem}\label{Rem:leftzero}
	Let $S$ be a nonempty left zero semigroup, that is, $xy=x$ for all $x,y\in S$.
	
	For $g\in S$, note that $D_{1,g} = D_{g,g} = D_{g,1} = \{g\}$. Indeed, if $a\in D_{1,g}$, then $ga = a = ag$, so $D_{1,g} = \{g\}$. The other two cases are proven similarly. We have $(g)\phi_{1,g} = gg1 = g$ and similarly, $(g)\phi_{g,1} = (g)\phi_{g,g} = g$. Therefore $\phi_{g,1} = \phi_{1,g} = \phi_{g,g}$. Thus other than the identity mapping on $S$, which is $\phi_{1,1}$, we may restrict our attention to those $\phi_{g,h}$ with both $g,h\in S$.
	
	For $a,g,h\in S$, if $a\in D_{g,h}$, then $g = gha = a$, and so $D_{g,h}(g) = \{g\}$. Thus $\phi_{g,h}$ is the map from $\{g\}$ to $\{h\}$ since $(g)\phi_{g,h} = hgg = h$. 
		
	For all $g,g,r,s\in S$, we have
	\[
	\phi_{g,h} \phi_{r,s} = \begin{cases} \phi_{g,s} & \text{ if } h = r\\
		\varnothing & \text{ otherwise}
	\end{cases}
	\]
	Therefore $\Inn(S) = \{\phi_{g,h}\,:\, g,h\in S\} \cup \{\id_S\}\cup  \{\varnothing\}$
\end{rem}

\begin{rem}\label{Rem:inverse}
The case where $S$ is an inverse semigroup is studied in detail in \cite{KS19}. It turns out that for any $g,h\in S^1$, $D_{g,h}\subseteq D_{g,g\inv}$. In that case, we may just work with the partial inner automorphisms $\phi_{g,g\inv}$ and for those, the inclusion
\eqref{Eq:aut_comp} is an equality. We then get a homomorphism $\Phi : S\to \Inn(S); g\mapsto \phi_{g,g\inv}$, whose kernel is precisely the central congruence of $S$. In particular, if $S$ is the symmetric inverse semigroup of partial injective transformations on a set $X$, then the homomorphism $\Phi$ is an isomorphism, and so $S\cong\Inn(S)$.
\end{rem}

\begin{example}
It is well known that non-isomorphic groups can have isomorphic automorphism groups (e.g., $Q_8$ and $S_4$ both have automorphism groups isomorphic to $S_4$). The same happens with partial inner automorphisms. The cyclic groups of order $2$ and $3$, both have the $2$-chain as their partial inner automorphism monoid (and the $2$-chain is isomorphic to its own partial inner automorphism monoid).
\end{example}

\begin{example}\label{Exm:centralizers}
In this example we show how partial inner automorphisms can help with understanding the natural conjugacy relation itself. 
An elementary observation in group theory is that if two elements
$a,b$ are conjugate, then every element of the centralizer $C_a$ of $a$
is conjugate to some element of the centralizer $C_b$ of $b$. This is not
true for $\cfn$, even in highly structured semigroups. Consider the
semigroup defined by this table:
\[
\begin{array}{c|cccccccc}
\cdot & e & r_1 & r_2 & s_1 & s_2 & s_3 & f & c \\
\hline
    e   & e   & r_1 & r_2 & s_1 & s_2 & s_3 & e   & s_1 \\
    r_1 & r_1 & r_2 & e   & s_3 & s_1 & s_2 & r_1 & s_3 \\
    r_2 & r_2 & e   & r_1 & s_2 & s_3 & s_1 & r_2 & s_2 \\
    s_1 & s_1 & s_2 & s_3 & e   & r_1 & r_2 & s_1 & e   \\
    s_2 & s_2 & s_3 & s_1 & r_2 & e   & r_1 & s_2 & r_2 \\
    s_3 & s_3 & s_1 & s_2 & r_1 & r_2 & e   & s_3 & r_1 \\
    f   & e   & r_1 & r_2 & s_1 & s_2 & s_3 & f   & c \\
    c   & s_1 & s_2 & s_3 & e   & r_1 & r_2 & c   & f
\end{array}
\]
This is a Clifford semigroup, that is, an inverse semigroup in which
the idempotents (in this case, $e$ and $f$) commute with all elements.
We see that this semigroup is a union (in fact, semilattice) of
the subgroups $A = \{e,r_1,r_2,s_1,s_2,s_3\}$ and $B = \{e,c\}$. Since $s_3^2 = e$, the identity element of $A$, we have that $A\subseteq D_{s_3,s_3}$. Now $(s_1)\phi_{s_3,s_3} = s_3 s_1 s_3 = s_2$, and thus $s_1\cfn s_2$. We see from the table that $C_{s_1} = \{ e, f, s_1, c \}$ and $C_{s_2} = \{ e, f, s_2 \}$. If $gh\cdot c = c = c\cdot gh$, then from the table, $gh = f$, and so $g=h=f$ or $g=h=c$. We compute
$c\phi_{f,f} = c$ and $c\phi_{c,c} = c$. Therefore the $\frn$-conjugacy
class of $c$ is $[c]_{\frn} = \{c\}$, and so $c$ is not $\frn$-conjugate
to any element of $C_{s_2}$.
\end{example}

An element $a$ of a semigroup $S$ is an \emph{epigroup element} (classically, a \emph{group-bound element}) if there exists a positive integer $n$ such that $a^n$ is contained in a subgroup of $S$.
The set of all epigroup elements of $S$ is denoted by $\Epi(S)$.
We can use the machinery above to show that in epigroups, or, more generally, in $\Epi(S)$, we can impose additional
restrictions on conjugators without loss of generality.

For $a\in \Epi_n(S)$, let $a^{\omega}$ denote the identity element of the group $\greenH$-class $H$ of $a^n$. Then $a^{\omega+1}:=aa^{\omega} = a^{\omega}a$ is in $H$. The \emph{pseudo-inverse} $a'$ of $a$ is $a'=(a^{\omega+1})^{-1}$, the inverse of $a^{\omega+1}$ in the group $H$ \cite[(2.1)]{Shevrin}. 

\begin{lemma}{\bf(\hskip-0.2mm\cite[Lem.~4.1]{ArKiKoMaTA})}
\label{Lem:xyx}
    Let $S$ be a semigroup and suppose that $uv, vu\in \Epi(S)$ for some
    $u,v\in S$. Then
    \begin{equation}\label{Eq:(uv)'u=u(vu)'}
          (uv)'u = u(vu)'\,.
    \end{equation}
\end{lemma}

Recall that elements $g,h$ of a semigroup $S$ are mutually inverse if $ghg=g$ and $hgh=h$.

\begin{theorem}\label{Thm:inv_conj}
	Let $S$ be a semigroup and let $g,h\in S^1$ satisfy $gh,hg\in \Epi(S^1)$. Then there exist mutually inverse $\bar{g},\bar{h}\in S^1$ such that $\phi_{g,h}\subseteq \phi_{\bar{g},\bar{h}}$.
\end{theorem}
\begin{proof}
Let $g,h\in S^1$. Setting
  \begin{equation}\label{Eq:inv_conj0}
  \bar{g} = (gh)^{\omega}g\quad\text{and}\quad \bar{h} = h(gh)',
  \end{equation}
we obtain:
  \begin{align}
  \bar{g}\bar{h} &= (gh)^{\omega}gh(gh)' = (gh)^{\omega},\label{Eq:inv_conj1}\\
  \bar{h}\bar{g} &= h(gh)'(gh)^{\omega}g = h(gh)'g \overset{\eqref{Eq:(uv)'u=u(vu)'}}{=} hg(hg)' = (hg)^{\omega},\label{Eq:inv_conj2}\\
  \bar{g}\bar{h}\bar{g} &= (gh)^{\omega}(gh)^{\omega}g = (gh)^{\omega}g = \bar{g},\notag\\
  \bar{h}\bar{g}\bar{h} &= h(gh)'(gh)^{\omega} = h(gh)' = \bar{h}\,.\notag
  \end{align}
  Therefore $\bar{g},\bar{h}$ are mutually inverse.

  Now assume $a\phi_{g,h} = b$, that is, $a\cfn b$ with $g,h$ as conjugators.
  We will now show that
  \begin{equation}\label{Eq:inv_conj3}
  (gh)^{\omega} a = a = a(gh)^{\omega} \quad\text{and}\quad (hg)^{\omega} b = b = b(hg)^{\omega}\,.
  \end{equation}
  Indeed, choose $n$ such that $(gh)^n (gh)^{\omega} = (gh)^{n+1} (gh)' = (gh)^n$. Then
  $a (gh)^{\omega} = a(gh)^n\cdot (gh)^{\omega} = a(gh)^n = a$. The other three equations in
  \eqref{Eq:inv_conj3} are proved similarly.

  Now we use \eqref{Eq:inv_conj0}, \eqref{Eq:inv_conj1}, \eqref{Eq:inv_conj2}, and \eqref{Eq:inv_conj3} in the following
  calculations:
  \begin{align*}
  a\bar{g} &= a(gh)^{\omega}g = ag = gb = g(hg)^{\omega}b = (gh)^{\omega}gb = \bar{g}b\,, \\
  \bar{h}\bar{g}\cdot b &= (hg)^{\omega}b = b\,,\,\,\text{and} \\
  a\cdot \bar{g}\bar{h} &= a(gh)^{\omega} = a\,.
  \end{align*}
  By Proposition \ref{Prp:alternatives}, $\bar{g},\bar{h}$ are conjugators for $a,b$, and thus
  $a\phi_{\bar{g},\bar{h}} = b$. This completes the proof.
\end{proof}

\begin{example}\label{Exm:strict}
In general, the conclusion of Theorem \ref{Thm:inv_conj} can be a strict inclusion. For example, consider the semigroup defined by the table
\[
\begin{array}{c|cccc}
\cdot & 1 & 2 & 3 & 4 \\
\hline
1     & 1 & 1 & 4 & 4 \\
2     & 2 & 2 & 3 & 3 \\
3     & 3 & 3 & 2 & 2 \\
4     & 4 & 4 & 1 & 1
\end{array}
\]
Set $g = 1$ and $h = 3$. Then $\bar{g} = 1$ and $\bar{h} = 2$. For $a = 1$, $b = 2$, we have $a\bar{g} = 1 = \bar{g}b$, $a\bar{g}\bar{h} = 1 = a$, $\bar{h}\bar{g}b = 2 = b$. Thus $a\cfn b$ with $\bar{g},\bar{h}$ as conjugators, so $a\phi_{\bar{g},\bar{h}} = b$. However, $agh = 3\neq a$ and so $a\not\in D_{g,h}$.
\end{example}

\begin{cor}\label{Cor:mi_conj}
If $a\cfn b$ in an epigroup $S$, then there exist mutually inverse conjugators for $a,b$.
\end{cor}

For the remainder of the paper, we will describe partial inner automorphisms in specific classes of semigroups: completely simple semigroups, full transformation monoids, and  $G$-sets, where $G$ is an abelian group. As will be seen, computing partial inner automorphisms in general presents a demanding challenge.

\section{The partial inner automorphism monoid of a completely simple semigroup}
\label{Sec:CS}

Every completely simple semigroup is isomorphic to a Rees matrix semigroup and so we assume at the outset that our semigroups have this form.

\begin{lemma}
	Let $\Gamma$ be a group, $I$ and $\Lambda$ two nonempty sets, and $P$ a $\Lambda\times I$ matrix with entries in $\Gamma$. Let $\mathcal{M}(\Gamma;I,\Lambda; P)$ be the Rees matrix semigroup induced by $\Gamma$, $I$, $\Lambda$ and $P$. Let $(G,g,\gamma), (H,h,\eta)\in \mathcal{M}(\Gamma;I,\Lambda; P)$. Then
	\[
	D_{(G,g,\gamma), (H,h,\eta)}\neq \varnothing\ \iff\ h = (p_{\eta, G} \ g\ p_{\gamma, H})^{-1}
	\]
	and
	\[
	D_{(G,g,\gamma), (H,(p_{\eta G}\  g\ p_{\gamma, H})^{-1},\eta)} =\{G\}\times \Gamma \times \{\eta\}.
	\]
\end{lemma}
\begin{proof}
	Regarding the equivalence, we start by proving the direct implication and the second equality. Let $(A,a,\alpha)\in \mathcal{M}(\Gamma;I,\Lambda; P)$ such that
	\[
	(G,g,\gamma) (H,h,\eta)(A,a,\alpha) = (A,a,\alpha) =  (A,a,\alpha) (G,g,\gamma) (H,h,\eta).
	\]
	Then $A=G$ and $\alpha = \eta$ so that
	\[
	D_{(G,g,\gamma), (H,h,\eta)}\subseteq \{G\}\times \Gamma \times \{\eta\}
	\]
	and hence the two sets are equal (by Lemma \ref{Lem:domain_facts}(4)). This proves the last equality in the statement of the lemma.
	
	Now, from $(G,g,\gamma) (H,h,\eta)(G,a,\eta)=(G,a,\eta)$, we get $g\ p_{\gamma, H} \ h\  p_{\eta, G}\ a = a$, that is,
	$h  = (p_{\eta, G} \ g\ p_{\gamma, H})^{-1}$. The direct implication is proved.
	
	For the converse implication, let $h=(p_{\eta, G} \ g\ p_{\gamma, H})^{-1}$ and $(G,a,\eta)\in {\mathcal M}(G;I,\Lambda; P)$. Then
	\[
	(G,g,\gamma) (H, p^{-1}_{\gamma, H}g^{-1}p^{-1}_{\eta, G},\eta)(G,a,\eta)=(G,a,\eta)
	\]
	and similarly
	\[
	(G,a,\eta)(G,g,\gamma) (H, p^{-1}_{\gamma, H}g^{-1}p^{-1}_{\eta, G},\eta)=(G,a,\eta).
	\]
	It is proved that $D_{(G,g,\gamma), (H,h,\eta)}\neq \varnothing$ and the lemma follows.
\end{proof}

Now we can state the main result of this section.

\begin{theorem}\label{Thm:RM}
	Let $\Gamma$ be a group, $I$ and $\Lambda$ two nonempty sets, and $P$ a $\Lambda\times I$ matrix with entries in $\Gamma$. Let $\mathcal{M}(\Gamma;I,\Lambda; P)$ be the Rees matrix semigroup induced by $\Gamma$, $I$, $\Lambda$ and $P$. Then the monoid $\Inn(\mathcal{M}(\Gamma;I,\Lambda; P))$ is generated by the following maps and corresponding inverses:
	\[
	\begin{array}{rccl}
		\phi_{(G,g,\gamma), (H, (p_{\eta, G} \ g\ p_{\gamma, H})^{-1},\eta)}:&\{G\}\times \Gamma\times \{\eta\}&\to& \{H\}\times \Gamma\times \{\gamma\}\\
		&(G,a,\eta)                                     &\mapsto & (H, (gp_{\gamma,H})^{-1}\ a\ (p_{\eta,G}\ g),\gamma),
	\end{array}
	\]
	for $g\in \Gamma$, $G,H\in I$ and $\gamma,\eta\in \Lambda$.
\end{theorem}

\section{The partial inner automorphism monoid of $T(X)$}
\label{Sec:mappingT(X)}
Computing the partial inner automorphisms of a given semigroup is a challenge in itself. We already observed that the symmetric inverse monoid is isomorphic to its inverse monoid of partial inner automorphisms (Remark \ref{Rem:inverse}). In this section, we describe the partial inner automorphism monoid $S=\Inn(T(X))$ for the full transformation monoid of a set $X$. It turns out that the structure of $S$ is essentially isomorphic to the combination of two components, one of which is the symmetric inverse monoid on $X$. The other component consists of bijections between partitions of $X$ with the same number of parts. In the same way that the partial composition operation of the symmetric inverse monoid is based on the intersection of an image and a domain, the operation of the second component is based on the join $\vee$ of two partitions.

In the above description, we write ``essentially" for two reasons. The two components are not entirely  independent, but are required to be compatible which each other in a natural way. In addition, further small adjustments are needed. In the case of finite $X$, the number of elements of $\Inn(T(X))$ that are affected by these adjustments are small relative to the size of $S$.

Throughout this section, we will blur the distinction between partitions and their corresponding equivalence relations.

\begin{theorem}\label{t:PI-sets}
Let $g,h\in T(X)$ and $D_{g,h}$ be as defined above, that is,
\[
D_{g,h}= \{ x\in T(X)\,:\, ghx = xgh = x \}\,.
\]
Then there exist a partition $P$ of $X$ and a partial section $I$ of $P$ such that $D_{g,h}$ consists of all transformations $t$ with $\ima t\subseteq I$ and $P\subseteq \ker t$. Moreover, $I, P$ can be chosen so that every singleton part $S$ of $P$ satisfies $S\subseteq I$.

The partial section $I$ is uniquely determined by $D_{g,h}$, and if $D_{g,h}$ contains more than one transformation, then $P$ is
uniquely determined by $D_{g,h}$ as well.

Conversely, suppose that $P$ is a partition of $X$ and $I$ is a partial section of $P$ such that all singleton parts of $P$ intersect $I$. Then there exist $g,h \in T(X)$ such that $D_{g,h}$ consists of all transformations $t\in T(X)$ with $\ima t\subseteq I$ and $P\subseteq \ker t$.

In the above cases, if $|I|\geq 2$, then $I,P$ uniquely determine $D_{g,h}$, while if $|I|\le 1$, then $I$ uniquely determines $D_{g,h}$.
\end{theorem}
\begin{proof}
Assume first that $g,h \in T(X)$, and let $D=D_{g,h}$. Clearly $D$ only depends on the product $p=gh$.

Let $I\subseteq X$ be the set of points fixed by $p$, and let $P$ be the collection of connected components of the function graph of $p$. In each part of $P$, there is at most a single point $x$ with $xp = x$, and so $I$ is a partial section of $P$. If $\{x\}$ is a singleton part of $P$ for some $x\in X$, then $xp=x$, and so $\{x\}\subseteq I$.

Let $t\in D_{g,h}$. Because $tp=t$, $p$ acts as the identity on the image of $t$ and so $t$ maps into $I$. Because $pt=t$, if $xp=y$, then $yt=x(pt)=xt$, and so $(x,y)\in\ker t$. It follows that the connected component of $x$ in the function graph of $p$ is contained in the kernel of $t$. Hence $P \subseteq \ker t$.

Conversely, if $t\in T(X)$ maps into $I$ and $P\subseteq\ker t$, it is straightforward to check that $pt=tp=t$, and so $t\in D$. It follows that $D$ consists of all $t$ with $\ima t\subseteq I$ and $P\subseteq \ker t$.

Now, let $I$ and $P$ be any set and partition that characterize $D$ in this way. Then $I$ is the union of all images of transformations in $D$, and hence is uniquely determined by $D$. If $|D|\geq 2$, then $|I|\geq 2$ and $|P|\geq 2$, the latter because $I$ is a partial section of $P$. Suppose that $P \in \{P_1, P_2\}$, where $P_1,P_2$ are two distinct partitions of $X$, each with at least two parts. Then without loss of generality $P_1$ is a refinement of a $2$-partition $P'$ of $X$ that does not contain $P_2$. Because $|I|\geq 2$, there exists a $t\in T(X)$ with $\ima t\subseteq I$ and $\ker t = P'\supseteq P_1$, but $P_2 \not\subseteq P'=\ker t$. It follows that $P$ is uniquely determined by $D$.

Now suppose that $P$ is a partition of $X$ and $I$ is a partial section of $P$ such that all singleton parts of $P$ are contained in $I$.

Let $g\in T(X)$ be the identity, and define $h\in T(X) $ as follows: if $x\in X$ is in a part $B$ of $P$ intersecting $I$, then let
$xh=y$ were $y$ is the unique element of $B\cap I$. If $B$ is a part of $P$ not intersecting $I$ then $|B| \ge 2$. Pick $b_1\neq b_2\in B$, and
let $b_1 h = b_2$, $xh = b_1$ for $x\in B\setminus\{b_1\}$. Applying the construction in the first part of the proof  to $D_{g,h}$, it is straightforward
to verify that we recover the sets $I$ and $P$. Hence $D_{g,h}$ contains all transformations $t$ with $\ima t\subseteq I$ and $P \subseteq \ker t$.

The final uniqueness result now also follows from the first part for $|I|\geq 2$, and is trivial for $|I|\leq 1$.
\end{proof}

For any $X$-partition $P$ and $I \subseteq X$, we will use the notation $D_{P,I}$ to refer to the set of $t\in T(X)$ with $\ima t\subseteq I$, $P\subseteq\ker t$, where we also include such $I,P$ in which $I$ is not a partial section of $P$, or for which not all singleton parts of $P$ intersect $I$.

%

\begin{lemma}\label{l:alpha}
Let $D_{g,h}= D_{P,I}$ and $D_{h,g}=D_{P',I'}$. Then $g|_I: I\to I'$ $,h|_{I'}:I'\to I$ are mutually inverse bijections.
\end{lemma}
\begin{proof}
The result is clear if $I=\varnothing$. Otherwise, pick $i \in I$, and define $t \in T(X)$ by
$[j]_Pt=j$ for $j \in I$, $xt=i$ otherwise. Clearly, $t \in D_{g,h}$ and  $ \ima t =I$.
Because $ght=t$, $\ima (ht)=I$, and because $htg \in D_{P',I'}$, we see that $g|_I$ maps into $I'$. Dually, $h_{I'}$ maps into $I$.

Because $t\in D_{g,h}$, $tgh=t$, and so $gh$ acts as the identity on the image $I$. Applying the
argument to a correspondingly constructed element $t'\in D_{h,g}$, we get that $hg$ is the identity on $I'$. The result follows.
\end{proof}

\begin{lemma}\label{l:beta}
Let $D_{g,h}= D_{P,I}$, $D_{h,g}=D_{P',I'}$ with $|I|\ge 2$ (and therefore $|I'| \ge 2$, by the previous lemma).

Then, $\hat{g}: P \to P'$, given by $[p]_P \hat{g} = [pg]_{P'}$, and $\hat{h}: P'\to P$, given by
$[p']_{P'} \hat{h} = [p'h]_{P}$, are well-defined mutually inverse bijections.

Moreover, for all $B \in P$, $B' \in P'$, we get $B\cap I = \varnothing \iff B\hat{g}\cap I'= \varnothing$
and $B' \cap I' = \varnothing \iff B'\hat{h} \cap I= \varnothing$.
\end{lemma}
\begin{proof}
Pick distinct $i,j \in I$, and $[p] \in P$. Define $t\in T(X)$ by $[p]_Pt=j, xt=i$ otherwise. Clearly, $t \in D_{g,h}=D_{P,I}$.
Because $j=pt=p(ght)$ we see that $p(gh) \in [p]_P$, and therefore $[p]_P(gh) \subseteq [p]_P$.

Suppose that $p_1,p_2 \in [p]_P$  are such that $[p_1g]_{P'} \neq [p_2 g ]_{P'}$. Let $t'\in D_{h,g}$ be a transformation that maps
$[p_1g]_{P'}, [p_2 g ]_{P'}$ to distinct elements $i'_1,i'_2 \in I'$ (such $t'$ clearly exists). Then $gt'h \in D_{g,h}=D_{P,I}$,
and therefore $i'_1h=p_1gt'h=p_2gt'h=i'_2h$, which contradicts the injectivity of $h|_{I'}$. It follows that $\hat{g}$ is well-defined.
A dual argument shows the corresponding claim for $\hat{h}$.

We already have seen that $p(gh)\in [p]_P$, and so $[p]_P \hat{g} \hat{h} = [p]_P$. As $[p]_P$ was arbitrary, we see that $\hat{g} \hat{h}$
acts as the identity on $P$. An analogous argument shows that $\hat{h} \hat{g}$ is the identity on $P'$, and hence $\hat{g}$ and $\hat{h}$
are inverse bijections.

The last claim follows from Lemma~\ref{l:alpha}.
\end{proof}

We now can derive a classification theorem for the generating elements  $\phi_{g,h}$ of the partial inner automorphism monoid.

\begin{theorem}\label{t:generators}
The partial inner automorphisms of $T(X)$ having the form $\phi_{g,h}$, and acting on more than one transformation are in bijective correspondence with the
tuples $(P, P',I, I', \alpha, \beta)$, where
\begin{itemize}
\item $P$ and $P'$ are partitions of $X$, with $|P|=|P'|$;
\item $I$ and $I'$ are partial sections, of $P$ and $P'$, respectively, with $|I|=|I'|\ge 2$, and intersecting all singleton sets of $P,P'$, respectively;
\item $\alpha: I \to I'$ is a bijection;
\item $\beta: P \to P'$ is a bijection extending the partial bijection between $P$ and $P'$ induced by $\alpha$
\end{itemize}
such that
\begin{itemize}
\item the domain of $\phi_{g,h}$ consists of all transformations $t \in T(X)$ with $\ima t \subseteq I$, $P \subseteq \ker t$;
\item the image of $\phi_{g,h}$ consists of all transformations $t \in T(X)$ with $\ima t \subseteq I'$, $P' \subseteq \ker t$;
\item
 Given $t$ in the domain of $\phi_{g,h}$, and $x \in X$, we have
$(x)(t\phi_{g,h})= i \alpha$, where $i\in I$ is the unique element in $(([x]_{P'})\beta ^{-1})t$.
\end{itemize}
The partial inner automorphisms of $T(X)$ having the form $\phi_{g,h}$ and acting on at most one transformation consist of all functions mapping one constant transformation on $X$ to another, and (for $|X|\neq 1$), the empty mapping.
\end{theorem}
\begin{proof}
We first consider the case of the partial inner automorphisms $ \phi_{g,h}$ whose domain contains more than one transformation. By Theorem~\ref{t:PI-sets}, $P,I,P',I'$ exist, have the stated properties and are uniquely determined by $D_{g,h}$ and $D_{h,g}$.
Set $\alpha= g|_I$, and $\beta = \hat{g}$, where $\hat{g}$ is defined as in Lemma~\ref{l:beta}. By Lemmas~\ref{l:alpha} and \ref{l:beta}, $\alpha$ and $\beta$ are bijections, and by its definition, $\beta$ extends the partial function on $P$ induced by $\alpha$.

Let $t\in \dom\phi_{g,h}=D_{P,I}$, and $x\in X$. By Lemma~\ref{l:beta}, $\beta^{-1}=\hat{h}$. Therefore $[x]_{P'}\beta^{-1} \in P$. As $t \in D_{P,I}$, $(([x]_{P'})\beta ^{-1})t$ contains a single element $i \in I$.

We now have that $x(ht) \in  ([x]_{P'}\hat{h})t=\{i\}$, and so $x(htg)= (x(ht))g=ig=ig|_I=i \alpha$, as required.

Now for any $i\in I$, let $c_i\in D_{P,I}$ be the constant function with image $i$. It follows from the above that $c_i \phi_{g,h}= c_{i\alpha}$,
and hence $\alpha$ is uniquely determined by $\phi_{g,h}$.

Finally suppose that $\beta, \beta':P\to P'$ are two bijections, that, together with some $\phi_{g,h},\alpha, P,I,P',I'$ satisfy the conditions of the theorem. Pick two distinct elements $i,j \in I$, and for each $B \in P$, let $t_B$ be the transformation with $Bt=\{i\}$, $xt=j$ for $x \notin B$. Let $x \in B \beta$, then
$x (t_B \phi_{g,h})=i \alpha$, as $([x]_{P'}\beta^{-1}) t_B=\{i\}$. Because $\alpha$ is injective, it follows that  $([x]_{P'}\beta'^{-1}) t_B=\{i\}$. From the definition of
$t_B$ this implies  $([x]_{P'}\beta'^{-1})= ([x]_{P'}\beta^{-1})$, and so $\beta^{-1}$ and $\beta'^{-1}$ agree on $B \beta$. As $B$ was arbitrary, we get $\beta=\beta'$.

The final claim about $\phi_{g,h}$ with $|D_{g,h}|\le1$ easily follows from Theorem~\ref{t:PI-sets}.
\end{proof}

We will now turn our attention to general elements of $\Inn(T(X))$.

\begin{defi}
Let $P,P'$ be partitions of $X$, and $\gamma: P\to P'$ a bijection. If $\bar{P} = \{B_i\}$ is a partition that refines to $P$, we define
$\bar{\gamma}$ on $\bar P$ by $(\cup B_i) \bar{\gamma} = \cup ((B_i) \gamma)$.
\end{defi}

It is clear that $\bar{\gamma}$ is well-defined, and that its image is a partition that refines to $P'$.

\begin{theorem}\label{t:general}
Let $\phi \in \Inn(T(X))$. Then there exist
\begin{itemize}
\item partitions $P, P'$ of $X$;
\item $I, I' \subseteq X$;
\item bijections $\alpha: I \to I'$, $\beta: P \to P'$ satisfying $[i]_P \beta = [i \alpha]_{P'}$ for all $i \in I$
\end{itemize}
such that
\begin{itemize}
\item the domain of $\phi$ consists of all transformations $t \in T(X)$ with $\ima t \subseteq I$, $P \subseteq \ker t$;
\item the image of $\phi$ consists of all transformations $t \in T(X)$ with $\ima t \subseteq I'$, $P' \subseteq \ker t$;
\item given $t$ in the domain of $\phi$, and $x \in X$, we have
$(x)(t\phi)= i \alpha$, where $i\in I$ is the unique element in $(([x]_{P'})\beta ^{-1})t$.
\end{itemize}
Moreover, if $\phi_1,\phi_2 \in \Inn(T(X))$ have corresponding parameters
\[
(P_1,I_1, P'_1, I_1', \alpha_1, \beta_1)\mbox{ and }(P_2, I_2, P'_2, I_2', \alpha_2, \beta_2)
\]
then $\phi_1\phi_2$ corresponds to
\[
((P'_1 \vee P_2) \bar \beta_1^{-1}, (I'_1 \cap I_2) \alpha_1^{-1}, (P'_1 \vee P_2) \bar \beta_2, (I'_1 \cap I_2) \alpha_2, \alpha_1\alpha_2, \bar \beta_1\bar \beta_2)\,,
\]
where $\alpha_1\alpha_2$ refers to the partial composition $\alpha_1|_{(I'_1 \cap I_2) \alpha_1^{-1}}\alpha_2$.
\end{theorem}
\begin{proof}
We will show the assertions by structural induction over the involved elements $\phi, \phi_1, \phi_2$. The beginning of the induction corresponds to those $\phi$ of the form $\phi_{g,h}$, and follows from Theorem \ref{t:generators} (in the cases with $|D_{g,h}|\le 1$, we can chose $P=P'=\{X\}, \beta=\id_{\{\{X\}\}}$).

Suppose the theorem holds for $\phi_1, \phi_2\in \Inn(T(X))$. Then $L:=\ima \phi_1\cap \dom \phi_2$ consists of all transformations $t$ with $\ima t\subseteq I'_1\cap I_2$ and $P_1'\vee P_2 \subseteq \ker t$. It is now straightforward to check that
\[
L \phi_1^{-1}=D_{(P'_1 \vee P_2) \bar \beta_1^{-1}, (I'_1 \cap I_2) \alpha_1^{-1}} \mbox{  and }L\phi_2=D_{ (P'_1 \vee P_2) \bar \beta_2, (I'_1 \cap I_2) \alpha_2}
\]
and hence these parameters define the domain and image of $\phi_1\phi_2$.

Let $i \in(I'_1 \cap I_2) \alpha_1^{-1}\subseteq I$, then
\[
[i]_{ (P'_1 \vee P_2) \bar \beta_1^{-1}}\bar \beta_1 \supseteq [i]_{P_1} \beta_1=[i \alpha_1]_{P'}\,,
\]
 and so
\[
[i]_{ (P'_1 \vee P_2) \bar \beta_1^{-1}}\bar \beta_1=[i \alpha_1]_{P'_1 \vee P_2} \supset [i \alpha_1]_{P_2}\,.
\]
Because $i \alpha_1 \in I'_1 \cap I_2 \subseteq I_2$, we get that
\[
[i]_{ (P_1' \vee P_2) \bar \beta_1^{-1}}\bar \beta_1\bar \beta_2 \supset [i \alpha_1]_{P_2}\beta_2 =[i \alpha_1\alpha_2]_{P_2'}\,.
\]
 Hence we get
\[
[i]_{ (P_1' \vee P_2) \bar \beta_1^{-1}}\bar \beta_1\bar \beta_2 =[i \alpha_1\alpha_2]_{(P'_1 \vee P_2) \bar \beta_2}\,,
\]
as required.

Let $t \in L \phi_1^{-1}$, and $x \in X$. Pick an element $y\in [x]_{(P'_1 \vee P_2)\bar \beta_2}\bar \beta_2^{-1}$. Because $\bar \beta_2^{-1}$ is injective, we have $[x]_{(P'_1 \vee P_2)\bar \beta_2}\bar \beta_2^{-1}=[y]_{P_1' \vee P2}.$
It follows that
\[
([x]_{(P'_1 \vee P_2)\bar \beta_2}(\bar \beta_1 \bar \beta_2)^{-1})t=([x]_{(P'_1 \vee P_2)\bar \beta_2}\bar \beta_2^{-1} \bar \beta_1^{-1})t=([y]_{P'_1 \vee P_2} \bar \beta_1^{-1})t=([y]_{P'_1} \beta_1^{-1})t\,,
\]
where the last equality holds  because the kernel of $t$ contains $(P'_1 \vee P_2) \bar \beta_1^{-1}$. By induction, this set contains a unique element $i$ such that $y (t \phi_1)=i \alpha_1$.

Also by induction,  $x ((t\phi_1) \phi_2) = j\alpha_2$, where $j$ is the unique element in
\[
([x]_{(P'_1 \vee P_2)\bar \beta_2}\bar \beta_2^{-1}) (t \phi_1)=([y]_{P_1' \vee P2})(t \phi_1)=\{y (t \phi_1)\}=\{i \alpha_1\}\,.
\]
Hence $x ((t \phi_1)\phi_2)= (i \alpha_1) \alpha_2$. Because $i \in ([x]_{(P'_1 \vee P_2)\bar \beta_2}(\bar \beta_1 \bar \beta_2)^{-1})$, the result follows.
\end{proof}

We can now obtain results about the structure of $\Inn(T(X))$. For a set $X$, let $A(X)$, $B(X)$ be the set of all bijections between subsets of $X$, and bijections on partitions of $X$, respectively. We say that
$\alpha\in A(X), \alpha:I\to I'$ and $\beta\in B(X), \beta: P \to P'$ are compatible, written $\alpha\approx \beta$, if $[i]_P \beta = [i \alpha]_{P'}$ for all $i \in I$.

Let $V(X) = \{(\alpha, \beta):\alpha\in A(X), \beta\in B(X), \alpha \approx \beta\}$. On $V(X)$ we define a binary operation
\[
(\alpha_1, \beta_1)(\alpha_2, \beta _2)=(\alpha_1 \alpha_2, \bar \beta_1 \bar \beta_2)\,,
\]
where $\bar \beta_i$ is as in Theorem \ref{t:general}, and  where we fix the domain of  $\alpha_1 \alpha_2$ [of $ \bar \beta_1 \bar \beta_2$] as the largest subset of $X$ [finest partition on $X$] for which these expressions are well-defined. It is easy to check that domains and images of  $\alpha_1 \alpha_2$ and $ \bar \beta_1 \bar \beta_2$ are given by the expressions from  Theorem \ref{t:general}.

It will follow from our results below that $V(X)$ with this operation is an inverse monoid. Because for every partial bijection $\alpha$ on $X$, there is a compatible $\beta$, the projection of $V(X)$ to its the first component is essentially the symmetric inverse monoid on $X$.

On $V(X)$, define a binary relation
\[
\theta = \Delta_{V(X)} \cup \{ ((\alpha, \beta_1),(\alpha, \beta _2)):
\alpha \in A(X),  |\dom \alpha| \le 1,\beta_1,\beta_2 \in B(X) \}\,.
\]
Clearly, $\theta$ is an equivalence relation, and because $\{(\alpha, \beta): |\dom \alpha|\le 1\}$ is an ideal of $V(X)$, $\theta$ is compatible with the operation on $V(X)$.
We set $W(X)= V(X)/ \theta.$ For $[(\alpha, \beta)]_\theta \in W(X)$ we will also use the short  notation $[\alpha, \beta]$.

\begin{theorem}\label{t:embed}
 Let $X$ be any set. For $\phi \in \Inn(T(X))$, let $\alpha_\phi, \beta_\phi$ be the  bijections associated with $\phi$ by \textup{Theorem \ref{t:general}}. Then $\varphi: \Inn(T(X))\to W(X)$, given by $\varphi(\phi)=[(\alpha_\phi, \beta_\phi)]_\theta$ is an embedding.

In particular, $\Inn(T(X))$ is isomorphic to the substructure of $W(X)$ generated by all elements of $W(X)$ that can be represented as $[(\alpha, \beta)]_{\theta}$ such that $\dom \alpha$ is a partial section of $\dom \beta$, and all singleton parts of $\dom \beta$ intersect $\dom \alpha$.
\end{theorem}
\begin{proof}
Our construction guarantees that $\varphi$ is a homomorphism, provided it is well defined.

Hence let $\phi \in \Inn(T(X))$, and $\alpha, \beta$ be the bijections associated with $\phi$.
Because $\dom \alpha$ and $\ima \alpha$ are the maximal images of all transformations in $\dom \phi$ and $\ima \phi$, respectively, they are uniquely determined by $\phi$.

For each $i\in \dom \alpha$, let $c_i$ be the constant function with image $i$. Then $c_i \in \dom \phi$, and $c_i \phi= c_{i \alpha}$. It follows that $\alpha$ is uniquely determined by $\phi$.

If $|\dom \alpha| \le 1$, then one $\theta$-class contains $(\alpha,\beta)$ for all choices of $\beta$. So assume otherwise, say $i,j \in \dom \alpha$.

Let $B\in \dom \beta$.
Because $\dom \phi$ contains the transformation $t_B$ that maps $B$ to $i$ and $X\setminus B$ to $j$, it follows that the parts of $\dom \beta$ are determined by all minimal kernel classes of transformations in $\dom \phi$. Hence $\dom \beta$ is unique, and similarly, we see that $\ima \beta$ is unique.

Finally, because $t_B \phi$ maps exactly $B \beta$ to $i \alpha$, we see that $\beta $ itself is uniquely determined. It follows that $\varphi$ is well-defined, and hence a homomorphism.

Moreover, for every $t \in \dom \phi$, and $x \in X$, we have $(x)(t\phi)= i \alpha$, where $i\in I$ is the unique element in $(([x]_{P'})\beta ^{-1})t$. Therefore $t\phi$ is uniquely determined by $\alpha, \beta$, and hence $\varphi$ is injective.

The final assertion follows from the description of the generators $\phi_{g,h}$ of $\Inn(T(X))$ in Theorem \ref{t:generators}, noting that in the case of $|\dom \alpha|\le 1$, we may always choose $\beta=\id_{\{\{X\}\}}$, in which case the representation $[\alpha,\beta]$ is as claimed.
\end{proof}

For a complete classification, it remains to determine the image of the embedding $\varphi$. We will have to distinguish between finite and infinite $X$. In the following, by the term ``generator'', we will mean an element of the form $\phi_{g,h} \varphi$.

\begin{theorem}\label{t:infinite}
Let $X$ be infinite. Then, $\Inn(T(X))$ is isomorphic to $W(X)$, and the embedding $\varphi$ from \textup{Theorem \ref{t:embed}} is an isomorphism.
\end{theorem}
\begin{proof}
By Theorem \ref{t:embed}, it suffices to show that $W(X)$ is indeed generated by all generators.
Let $I \subseteq X$, and $P$ be a partition $X$. Clearly, $\id_I \approx \id_P$. We first show that $[(\id_I, \id_P)]_\theta$ is in the image of $\varphi$.

Choose a bijection $\sigma: X \to X^2$.  Let $P_1$ be the singleton partition on $X$, $P_1'= \{(\{x\} \times X)\sigma^{-1}: x \in X\}$, and define
$\alpha_1:X \to (\Delta_X) \sigma^{-1} $,
 $\beta_1: P_1 \to P_1'$ by
$x \alpha_1= (x,x) \sigma^{-1}$, $ \{x\} \beta_1= (\{x\}\times X) \sigma^{-1}$. It is straightforward to check
 that $[\alpha_1, \beta_1]$ is a generator.

Next let $\alpha_2$ and $\beta_2$ be the identities on $\{(x,x) \sigma^{-1}:x \in I\}$ and $P_1'$, respectively. Because $P_1'$ does not contain any singleton blocks, $[\alpha_2, \beta _2]$ is once again a generator.

Let $\beta_3$ be the identity on the partition $P_3$ consisting of all sets of the form $\{(x,y), (y,x)\}\sigma^{-1}$ for $x,y \in X$ with $[x]_P=[y]_P$, and singletons otherwise. Moreover, let $I_3$ be the union of all singleton sets in $P_3$ and $\alpha_3=\id_{I_3}$. Once again, $(\alpha_3, \beta_3)$ is a generator.

Finally, let $\alpha_4= \alpha_1^{-1}$, $\beta_4=\beta_1^{-1}$. We claim  $[(\id_I, \id_P)]_\theta =\Pi_{i=1}^4[(\alpha_i, \beta_i)]_\theta$.
Let $x \in I$, then
\[
x \alpha_1 \alpha_2 \alpha_3 \alpha_4= ((x,x) \sigma^{-1}) \alpha_2 \alpha_3 \alpha_4
=((x,x) \sigma^{-1}) \alpha_3 \alpha_4= ((x,x) \sigma^{-1})\alpha_4= x\,.
\]

If $x \notin I$, then $\alpha_2$ is undefined at $x \alpha_1=((x,x) \sigma^{-1})$. Hence $\alpha_1 \alpha_2 \alpha_3 \alpha_4= \id_I$.
Let $B \in P$, and $C \subseteq B$. Then,
\[
C \bar \beta_1\bar \beta_2 \bar \beta_3\bar \beta_4= ((C\times X) \sigma^{-1})\bar \beta_2 \bar \beta_3\bar \beta_4= ((C\times X) \sigma^{-1}) \bar \beta_3\bar \beta_4= ((B\times X) \sigma^{-1}) \bar \beta_4=B\,.
\]
From this it follows that the domain of $\bar \beta_1\bar \beta_2 \bar \beta_3\bar \beta_4$ is indeed $P$ (as opposed to a refinement), and that $\bar \beta_1\bar \beta_2 \bar \beta_3\bar \beta_4$ acts as the identity. Hence $[(\id_I, \id_P)]_\theta$ is in the image of $\varphi$, as claimed.

For the general case, let  $[\alpha, \beta]_\theta \in W(X)$ be arbitrary. Construct $[\alpha', \beta']$ as follows:
If $B_i \in \dom \beta$ intersects $\dom \alpha$, choose a partition $P_{B_i}$ of $B_i$ that contains exactly one element of $\dom \alpha$ in each part, and let $\dom \beta'$ be the union of the $P_{B_i}$, together with all $B \in \dom \beta$ not intersecting $\dom \alpha$. Note that $\dom \beta'$ is a refinement of $\dom \beta$.
Let $\ima \beta'$ be the  refinement obtained from $\ima \beta$ in the same way. If $B'_i\in \dom \beta'$ contains a (unique) element $i \in \dom \alpha$, then let $B'_i \beta'= [i \alpha]_{\ima \beta'}$, otherwise, set $B'_i \beta'= B_i' \beta$.
If $B_i \in \dom \beta $ does not intersect $\dom \alpha$, choose an element $b_i \in B_i$. Let $\dom \alpha'$ be obtained from $\dom \alpha$ by adjoining all the elements $b_i$. Similarly enlarge $\ima \alpha$ to $\ima \alpha'$ by choosing one element from each $B_i \in \ima \beta$ that does not intersect $\ima \alpha$. Now let $x \alpha'$ be the unique element in $\ima \alpha' \cap [x]_{\dom \beta'}\beta'$.

Then $[\alpha', \beta'] $ is a generator. Since $[\id_{\dom \alpha}, \id_{\dom \beta}]\in \ima \varphi$, this also holds for
$[\id_{\dom \alpha}, \id_{\dom \beta}] [\alpha', \beta']$. A straightforward check shows that this product is $[\alpha,\beta]$, and the result follows.
\end{proof}

\begin{theorem}\label{t:finite}
Let $X$ be finite, and $[\alpha,\beta]_{\theta} \in W(X)$. If $|\dom \alpha| \ge 2$, then
$[\alpha,\beta]_{\theta} \in \ima \varphi$ if and only if one of the following holds:
\begin{itemize}
\item[\textup{(1)}] $\dom \alpha=X$ and $\dom \beta$ is the partition of $X$ into singletons;
\item[\textup{(2)}] there exists $B \in \dom \beta$ with $|B| \ge 2$, $B \not\subseteq \dom \alpha$.
\end{itemize}
If $|\dom \alpha| \le 1$, then $[\alpha,\beta]_{\theta} \in \ima \varphi$, unless $|X|=1$ and $\dom \alpha= \varnothing$.
\end{theorem}
\begin{proof}
Suppose first that $|\dom \alpha|\geq 2$. If $[\alpha,\beta]$ satisfies condition 1, then it is a generator, and hence in the image of $\varphi$ (in fact its preimage will be a unit of $T(X)$).

So assume that there exists a set $B\in \dom \beta$ with $|B|\geq 2$, $B\not\subseteq \dom \alpha$.  Let $I = \dom \alpha, P= \dom \beta$. As in the infinite case, we first show that $[(\id_I, \id_P)]_\theta$ is in the image of $\varphi$.

Enumerate $X$ as $x_1, x_2, \dots, x_m$, such that the parts of $P$ correspond to consecutive index ranges in $\{1, \dots,m\}$, with $x_m \in B \setminus I$. We will use three different types of generators to obtain $[\id_I, \id_P]$.

For $J\subseteq I\setminus \{x_m\}$, let $Q_J$ be the partition with part $J \cup \{x_m\}$, and singletons otherwise. If $J=\{x_j\}$, we will just write $Q_{x_j}$.
We set $k_j=[\id_{I\setminus \{x_m\}},\id_{Q_{x_j}}]$, and $l_J =[\id_{I\setminus J},\id_{Q_J}]$. Moreover, let $\beta_j: Q_j \to Q_{j+1}$ be defined by $\{x_j, x_m\} \beta_j= \{x_j\}$, $\{x_{j+1}\} \beta_j= \{x_{j+1},x_m\}$, and the identity otherwise.
Set $s_j= [\id_{I \setminus \{x_m\}}, \beta_j]$. It is easy to check that all $k_j, l_J,$ and $s_j$ are generators.

Let $C_1,\ldots, C_r=B$ be the parts of $P$, in the order of their index ranges. For each $C_i=\{x_{d_i},\ldots, x_{e_i}\}$, $i=1,\ldots, r-1$, let
$J_i= C_i \setminus I$, and set
$p_i= k_{d_i} k_{d_i+1}\ldots k_{e_i} l_{J_i}s_{e_i}$.
For $C_r=B=\{x_{d_r},\ldots, x_m\}$, let $J_r= B \setminus I$ and set $p_r=k_{d_r} k_{d_r+1} \ldots k_{m-1} l_{J_r}$.

We leave it up to the reader to confirm that $[\id_I, \id_P]=p_1 \cdots p_r$.
We now can show that $\ima \varphi$ contains any $[\alpha, \beta]$ with $\dom \alpha =I, \dom \beta =P$ exactly as in the infinite case in Theorem \ref{t:infinite}.

For the converse, suppose that $a=[\alpha,\beta]_{\theta} \in \ima \varphi$, say $a=g_1\cdots g_n$ for some generators $g_i=[\alpha_i, \beta_i]$.

If $\dom \alpha=X$, then by finiteness, $\dom \alpha_i=X$ for all $i$, and hence (as the $g_i$ are generators), $\dom \beta_i$ is the partition into singletons. From this, we get that $\dom \alpha=X$ and $\dom \beta$ is the partition of $X$ into singletons, as well.

Let $\dom \alpha\neq X$. We may assume that the number of generators $n$ is the smallest possible. If $\dom \alpha_1=X$, then it is easy to see that $g_1g_2$ is a generator as well (note that this requires finiteness, which forces $g_1 \varphi ^{-1}$ to be a unit of $T(X)$).

Hence by minimality, $\dom \alpha_1 \neq X$. As $g_1$ is a generator, it follows that $\dom \beta_1$ contains a set $B'$, $|B'| \ge 2$ with $B'\not\subseteq \dom \alpha_1$. But then $\dom \beta$ contains a set $B$ with $B'\subseteq B$ and $\dom \alpha \cap B'\subseteq \dom \alpha_1$. It follows that $B$ satisfies the criteria in condition (2).

If $|\dom \alpha|=1$ then $[\alpha, \beta]_\theta=[\alpha,\id_{\{X\}}]_\theta$, which is a generator. If
$|\dom \alpha|=0$ and $|X| \neq 1$,  then $[\alpha,\beta]$, which is the empty mapping, is the generator $[\varnothing,  \id_{\{X\}}]$.
Conversely, if $|X|=1$, then $\Inn(T(X))$ only contains the trivial full automorphism.
The result follows.
\end{proof}


\section{The partial inner automorphism monoid of the endomorphism monoid of a finite abelian $G$-set}
\label{Sec:mappingG-set}
\label{rightleft2}

Let $G$ be a group with identity $e$. A (left) $G$-set is a set $X$ together with an action $\cdot: G \times X \to X$, such that for all $k,l \in G$, $ x\in X$, we have
$(kl)\cdot x= k\cdot (l \cdot x)$ and $e \cdot x=x$.  Throughout this section, we will assume that $X$ and $G$ are finite, and that $G$ is abelian.

A \emph{$G$-endomorphism} of $X$ is a function $f\in T(X)$ such that $f(k\cdot x)= k \cdot f(x)$ for all $k \in G, x \in X$. Exceptionally in \S\ref{rightleft2}, we will compose transformations from right to left, that is, $(f\circ g)(x)=f(g(x))$. This change is for compatibility with the common practice of using left instead of right $G$-sets.
With this convention, the set $\End_G(X)$ of all $G$-endomorphisms of $X$ is a submonoid of $T(X)$.

If $x,y \in X$ lie in the same $G$-orbit $O$, then it follows from the commutativity of $G$ that $x$ and $y$ have the same point stabilizer $G_x=G_y$. We set $G_O=G_x$ for any $x \in O$. 

Let $f \in \End_G(X)$. The following facts about $f$ are easily checked:
\begin{itemize}
\item $f$ maps $G$-orbits to $G$-orbits;
\item if $f(x)=y$ then $G_x \le G_y$;
\item if $x\in X$ lies in the $G$-orbit $O$, then $f(x)$ determines $f(y)$ for all $y \in O$ as $f(k\cdot x)=k \cdot f(x)$.
\end{itemize}
We use $G_B$ to refer to the set-wise stabilizer of $B$, for any $B \subseteq X$.

We will adopt our results from \S\ref{Sec:mappingT(X)} to the case of the endomorphism monoid of a finite $G$-set $X$, where $G$ is a finite abelian group.

A partition $P$ is called $G$-invariant if the corresponding equivalence relation $\tau$ on $X$ satisfies $(x,y) \in \tau \Rightarrow (k \cdot x, k \cdot y) \in \tau$ for all $x,y \in X, k \in G$. Clearly, if $\phi \in \End_G(X)$, then $\ker \phi$ corresponds to a $G$-invariant partition, and if $P$ is a $G$-invariant partition, then $G$ has an induced action on $P$. We will call a $G$-invariant partition $P$
on a $G$-set $X$ \emph{non-null} if for each $B_i \in P$ there exists $x_i \in B_i$ such that $G_x \le G_{x_i}$ for all $ x \in B_i$. In this case, we set  $G^{B_i}=G_{x_i}$.

We will often use the following construction, the correctness of which is easy to check: after picking representatives $B_i$ for the $G$-orbits on $P$, we may obtain a set of representatives $\{x_j\}$ for the $G$-orbits on $X$ with $\{x_j\}\subseteq \bigcup B_i$ by separately picking representatives for the $G$-orbits on $X$ that intersect each $B_i$, and then taking the union all such sets. After we choose a mapping $h:\{x_j\} \to X$ such that $G_{x_j} \le G_{h(x_j)}$ for all $j$, we can extend $h$ to a unique and well-defined $G$-endomorphism of $X$ by setting $h(l \cdot x_j) = l \cdot h(x_j)$.
If, in addition, $h(x_j)=h(x_{j'})$ whenever $[x_j]_P=[x_{j'}]_P$, and $G_{[x_j]_P} \le G_{h(x_j)}$ for all $x_j$, then $P \subseteq \ker(h)$.

Given a partition $P$ of $X$, and a subset $I \subseteq X$, we say that $I$ is  \emph{accessible} from $P$ if for every $B \in P$, there exists an $i \in I$ with $G_B \le G_i$, and \emph{non-accessible} for $P$ otherwise. We say that an element $i$ of such a set $I$  is a $\emph{sink}$, if $i$
is the unique member of $I$ satisfying $G_i=G$.

We will now prove results for the partial automorphism monoid of $\End_G(X)$, following closely the outline of \S\ref{Sec:mappingT(X)} on the automorphism monoid of $T(X)$, while also relying on some of that section's results.
Throughout, we will blur the distinction between partitions and their corresponding equivalence relations.

\begin{theorem}\label{t:tauI-sets}
Let $g,h \in \End_G(X)$ and
\[
D_{g,h}= \{x \in \End_G(X)\,:\, ghx=xgh=x\}\,.
\]
Then there exist a non-null $G$-invariant partition  $P$ of $X$ and a union of $G$-orbits $I$ of $X$ (possibly empty),  such that:
\begin{itemize}
\item[\textup{(1)}] each $P$-class contains at most one element of $I$;
\item[\textup{(2)}] for each $i \in I$, if $x \in [i]_P$, then $G_x \le G_i$;
\item[\textup{(3)}]  if $i \in X$ satisfies $G_x < G_i$ for all $x \in [i]_p \setminus \{i\}$, then $i \in I$;
\item[\textup{(4)}] for each $B \in P$, the quotient $G_{B}/G^B$ is cyclic;
\item[\textup{(5)}]  $D_{g,h}$ consists of all $G$-endomorphisms $\phi$ with $\ima \phi\subseteq I$ and $P \subseteq \ker \phi$.
\end{itemize}

Conversely, suppose that $P$ is a non-null $G$-invariant partition of $X$ and $I$ is a union of $G$-classes  such that \textup{(1)} to \textup{(4)} hold. Then there exist $g,h\in \End_G(X)$ such that \textup{(5)} holds.

In the above constructions, $D_{g,h}\neq \varnothing$ if and only if $I$ is accessible from $P$, in which case $I$ is uniquely determined by $D_{g,h}$.
\end{theorem}
\begin{proof}
Assume first that $g,h \in \End_G(X)$, and let $D=D_{g,h}$. Clearly $D$ only depends on the product $p=gh$. As $g,h \in \End_G(X) \subseteq T(X)$, we have that
\[
D = \{t \in T(X)\,:\, ght=tgh=t\} \cap \End_G(X)\,.
\]
Hence, \textup{(5)} holds if we define $I, P$ as in Theorem \ref{t:PI-sets}, that is, $I$ is the set of fixed points of $p$, and $P$ is the partition into connected components of the function graph of $p$.  As $p$ is a $G$-endomorphism, $I$ is a union of  $G$-orbits and $P$ is $G$-invariant. Moreover, that $P$ is non-null and properties \textup{(1)} to \textup{(3)} follows from this description.

To show \textup{(4)}, let $B \in P$, and $x$ any element in the cyclic part of $B$ (as a connected component of the function graph of $p$), so that $G^B=G_x$. Let $O$ be the $G$-orbit containing $x$, then $G_B=G_{O \cap B}$.

Let $n$ be the smallest positive integer such that $p^n(x)$ is in $O$, say $p^n(x)= l\cdot x$ for $l \in G$. Then $p^m(x) $ is in $O$ if and only if $n$ divides $m$, in which case $p^m(x)= l^{m/n} \cdot x$. Moreover, let $y \in O \cap B$, with $k' \cdot y$ for some $j' \in G$.  Because $x$ is on the cyclic part of $B$, there exist an $m' >0$, such that $p^{m'}(y)=x$. This implies that $p^{m'(|k'|-1)}(x)=y$. Therefore,
$y$  is in the cycle of $B$, and hence $n$ divides $m'(|k'|-1)$, and $y=l^{m'(|k'|-1)/n} \cdot x$. Thus $\langle lG^B\rangle=\langle lG_x\rangle=G_B/G_x=G_B/G^B$.

Conversely, suppose that $P$ is a non-null $G$-invariant partition of $X$ and $I$ is a union of $G$-classes  such that \textup{(1)} to \textup{(4)} hold. Let $g\in \End_G(X)$ be the identity, and define $h\in \End_G(X) $ as follows:

Choose sets of representatives $\{B_i\}$ and $\{x_{j}\}\subseteq \bigcup B_i$ as outlined in the introduction of this section. Let $B \in P$ such that $i \in B$ for some $i \in I$. As $i$ is unique in $I \cap P$, and $I$ is closed under $G$, it follows that $G_B = G_i$. Hence if we set $h(x_j) =i$ for all
 $x_j \in B$, the resulting induced function will satisfy $h(B)=\{i\}$, and every $B'$ from the orbit of $B$ will also map to an element of $I$. Suppose instead that $B \cap I= \varnothing$. Without loss of generality, let $x_{1}, \dots, x_n \in B$ be all the representatives with $G_x=G^B$, and let $l\in G$ be such that $lG^B$ generates $G_B/G^B$. Set $h(x_i)=x_{i+1}$ for $i=1, \dots, n-1$, and
$h(x_n)=l \cdot x_1$. For any other representative $x_j \in B$, set $h(x_j)=x_1$. Then the resulting function $h$ has $B$ as one of the connected components of its function graph, and the same holds for any $P$-part from the $G$-orbit of $B$.

It is now straightforward to check that $gh$ has $I$ as its sets of fixed points, and $P$ as its partition into connected components of the function graphs. By the first part of the theorem, 5.\ holds for this $g$ and $h$.

Now consider any pair $(P,I)$, satisfying all the given conditions. If $I$ is non-accessible from $P$, and this failure is witnessed by $B \in P$, then there is no $G$-endomorphism that can map $B$ to an element of $I$, and hence $D_{g,h}$ is empty.
For the converse, let $I$ be accessible. Choose a set of representatives $\{B_i\}$ and $\{x_{j}\}\subseteq \bigcup B_i$ as above. For each $B_i$ pick a $k_i \in I$ such that $G_{k_i}$ contains $G_{B_i}$, and set $t(x_j)=k_i$ for all $x_j \in B_i$. It is straightforward to check that the extension of $t$ to all of $X$ is a $G$-endomorphism with  $\ima t\subseteq I$ and $P \subseteq \ker t$.

Finally assume that  $P,I$ satisfy the conditions \textup{(1)} to \textup{(5)}  for some non-empty $D_{g,h}$, noting that this implies that $I$ is accessible from $P$.  We claim that $I$ is the union of all images of $G$-endomorphisms in $D_{g,h}$ and hence uniquely determined by $D_{g,h}$. Clearly $I$ contains this union.

For the converse inclusion, given $i \in I$, we may choose a set of orbit representatives that includes $i$. We claim that we can define a $G$-endomorphism $t_i$ with $t_i(i)=i$ and $P \subseteq \ker t_i$.
Indeed, as $[i]_P$ contain only one element of $I$ and $I$ is closed under $G$, it follows that  $G_{[i]_P}=G_i$, which allows the assignment $t_i(i)=i$. Because $I$ is accessible from $P$, and because of our other conditions, we may define $t_i$ on the other representatives such that  $t_i(x_j) \in I$, $G_{[x_j]_P} \le G_{t_i(x_j)}$, and $t_i(x_j)=t_i(x_{j'})$ whenever $[x_j]_P=[x_{j'}]_P$. This way, we obtain the desired $G$-endomorphism $t_i$. The uniqueness of $I$ follows.
\end{proof}

For any non-null $G$-invariant partition $P$ of $X$ and union of $G$-orbits $I \subseteq X$, we will use the notation $D_{P,I}$ to refer to the set of $\phi\in \End_G(X)$ with $\ima \phi\subseteq I$, $P\subseteq\ker \phi$, where we also include such $I,P$ which do not satisfy any of the other condition of Theorem \ref{t:tauI-sets}.

We will next address  to which extent $P$ is determined by $D_{g,h}$.

\begin{lemma}\label{l:Pmerge}
Suppose that $P,I$ satisfy the conditions of  Theorem \ref{t:tauI-sets}, and that $D_{P,I}\neq \varnothing$. For each $B\in P$, let
\[
G'_B= \bigcap_{i \in I, G_B \le G_i}G_i\,.
\]
Define the binary relation $\sim_P$ on $P$ containing exactly the following pairs:
\begin{enumerate}
\item[(a)] $(B, l\cdot B)\in \sim_P$, for all $B \in P$ and $l \in G'_B$;
\item[(b)] if $I$ has a sink $i'$, $(B,B')\in\sim_P$ whenever $G_B, G_{B'} \not\le G_i$, for all $ i \in I\setminus\{i'\}$.
\end{enumerate}
Then $\sim_P$ is an equivalence relation. Let $P'$ be the partition of $X$ induced by the equivalence classes of $\sim_P$. Then $P'$ is a $G$-invariant, non-null partition of $X$, such that $D_{P,I}=D_{P',I}$. Moreover,
$(P',I)$ satisfy conditions \textup{(1)}-\textup{(3)} of \textup{Theorem \ref{t:tauI-sets}}.

Conversely, suppose that $P''$ is a $G$-invariant partition such that $D_{P,I}=D_{P'',I}$. Then $P''$ is a refinement of $P'$.
\end{lemma}
\begin{proof}
It is clear that $\sim_P$ is an equivalence relation and compatible with the action of $G$.

The elements in $B$ and $l \cdot B$ have the same set of stabilizers, while any block $B$ satisfying condition (b) is contained in $[i']_{P'}$ with $G_{i'}=G.$ Hence $P'$ is non-null.

As $P$ is a refinement of $P'$, we have $D_{P',I}\subseteq D_{P,i}$. For the other inclusion, let $t \in D_{P,I}$, $B \in P$, and $i \in I$, with $t(B)=\{i\}$. If  $l \in G'_B$ and $B'=l \cdot B$, then
 $t(B')=t(l\cdot B)=l\cdot t(B)=l \cdot \{i\}=\{i\}$. Hence $t(B)=t(B')$ for all $t \in D_{P,I}$.
If $I$ has a sink $i'$ and $B$ and $B'$ are as in (b), we obtain $t(B)=t(B')= \{i'\}$ for all $t \in D_{P,I}$, as $\{i'\}$ is the only possible image for $B$ and $B'$. Hence $t(B)=t(B')$ whenever $B\sim_PB'$.
It follows that $D_{P',I}=D_{P,I}$.

If $i\in I \cap B$ for  some $B \in P$, then $G'_B=G_B=G_i$. Hence $[i]_P=[i]_{P'}$ for all $i \in I \setminus \{i'\}$. This immediately implies  condition \textup{(1)} from Theorem  \ref{t:tauI-sets}. Condition (b) follows in connection with $G_{i'} =G$. Finally, the premise of condition \textup{(3)} from Theorem  \ref{t:tauI-sets} only holds in a $P'$-block $B$ if it also holds in a $P$-block $B$ with $B \subseteq B'$. Condition (3) follows because it was true for $P$.

Conversely, suppose that $B,B'\in P$ are such that  $B \not\sim_P B'$. If $B$ and $B'$ do not lie in the same $G$-orbit on $P$, then there exists $t \in D_{P,I}$ with $t(B) \neq t(B')$ as long as at least one of $B,B'$ has two possible images in $I$. As (b) does not hold for $B,B'$, this is always the case.

If instead $B'= l \cdot B$ for some $l \in G$, then the falsehood of (1) guarantees that there is an $i\in I $ with $G_B \le G_i$ and $l \notin G_i$. Hence there exists a $t \in D_{P,I}$ with $t(B)=\{i\} \ne\{l\cdot i\}= t(B')$.

If a $G$-invariant partition $P''$ satisfies  $D_{P,I}=D_{P'',I}$, we my assume without loss of generality that it has $P$ as a refinement. The results from the previous two paragraphs show that $P''$ does not merge any $P$-blocks not satisfying (a) or (b). Hence $P'$ is the coarsest  such partition.
\end{proof}

We remark that condition \textup{(4)} of Theorem \ref{t:tauI-sets} might not be true for $(P',I)$. This implies that there may not exist $g,h \in \End_G(X)$, such that $P'$ consists of  the connected components of the function graph of $gh$.

Clearly, $P'$ as constructed in Lemma \ref{l:Pmerge} is unique for each pair $(P,I)$ satisfying the conditions of the lemma. Hence, we set $\overline{(P,I)}=(P',I)$. In addition, we set $\overline{(P,I)}=(\{X\}, \varnothing)$ whenever $I $ is not accessible from $P$. Clearly, $\overline{\overline{(P,I)}}=\overline{(P,I)}$. We will refer to $\overline{(P,I)}$  as the \emph{standard pair} for $(P,I)$.

More generally, we will refer to $(P,I)$ as a \emph{standard pair} if either $(P,I)=(\{X\},\varnothing)$ or the following hold:
\begin{enumerate}
\item $P$ is a $G$-invariant non-null partition of $X$;
\item $I\subseteq X$ is a union of $G$-orbits;
\item $I$ is accessible from $P$;
\item conditions \textup{(1)-(3)} of Theorem \ref{t:tauI-sets} hold;
\item $\overline{(P,I)}=(P,I)$.
\end{enumerate}
Given a standard pair $(P,I)$ it might not be possible to use the construction from the second half of Theorem  \ref{t:tauI-sets} due to the lack of condition \textup{(4)}. The next result shows when a standard pair arises from two $G$-endomorphisms.

\begin{lemma}\label{l:spairs}
Let $(P,I)\neq (\{X\},\varnothing)$ be a standard pair. Then $g,h \in \End_G(X)$ with $D_{g,h}=D_{P,I}$ exist if and only if for every $B\in P$ not intersecting $I$, there exists an $l_B \in G_B$ such that
\begin{equation}\label{eqt}
\langle l_B, G^B\rangle \le G_i \Rightarrow G_B \le G_i
\end{equation}
for all $i \in I$.

If $(P,I)=(\{X\},\varnothing)$ then there exist $g,h \in \End_G(X)$ such that $D_{g,h}=D_{P,I}$ if and only if
\[
|\{x\in X: G_x=G\}|\neq 1\,.
\]
\end{lemma}
\begin{proof}
Let $(P,I)\neq (\{X\},\varnothing)$, and assume that $l_B$ exists for every $B$ not intersecting $I$. We want to use the second part of Theorem  \ref{t:tauI-sets} after first constructing a refinement $P'$ of $P$ that satisfies both condition \textup{(4)} of Theorem  \ref{t:tauI-sets} and $D_{P,I}=D_{P'\!,I}$.

Let $B \in P$. If $B$ intersect $I$, then it follows from conditions \textup{(1)} and \textup{(2)} of Theorem \ref{t:tauI-sets} that  $G_B/G^B$ is trivial. If $B$ does not intersect $I$, consider a set of representatives $R$ for the $G$-orbits intersecting $B$, and the relation
\[
\theta_B=\{(l\cdot x_j, l' \cdot x_{j'}): l,l' \in G_B,  x_j,x_{j'}\in R, l l'^{-1} \in \langle l_B, G^B\rangle\}
\]
on $B$. Because $G_x \le G^B$ for all $x \in B$, $\theta_B$ is a well-defined equivalence relation on $B$. Let $P_B$ the corresponding partition of $B$, and choose an $x_B \in B \cap R$ with $G_{x_B}=G^B$. For each $B' \in P_B$, if $l \cdot x_j \in B'$, for some $l \in G$ and $x_j \in R$, then $l \cdot x_B \in B'$ and hence $G^{B}=G_{x_B}= G_{l \cdot x_B}=G^{B'}$.
Moreover, we have that $G_{B'}= \langle l_B, G^B\rangle=\langle l_B, G^{B'}\rangle$ and hence $G_{B'}/G^{B'}$ is cyclic. Now if $t \in \End_G(B)$,
(\ref{eqt}) shows that if $t(B')=\{i\}$ for some $i \in I$, then $t(B)=\{i\}$.

Repeating the above construction for each $B \in P$ not intersecting $I$, we obtain a refinement $P'$ of $P$, such that for each $B' \in P'$, the quotient $G_{B'}/G^{B'}$ is cyclic, and such that $D_{P,I}=D_{P',I'} $. We can now use the second part of   Theorem  \ref{t:tauI-sets} to construct the required $g,h$.

Conversely, if $g,h $ exist with $D_{g,h}= D_{P,I}$, then $D_{P,I}=D_{P',I}$ where $P'$ consist of the connected components of the function graph of $gh$, by the proof of  Theorem \ref{t:tauI-sets}. The theorem moreover shows that for each $B'\in P'$,  $G_{B'}/G^{B'}$ is cyclic. Let $l_{B'}$ a cyclic generator of this group, so that $G_{B'}=\langle l_{B'},G^{B'}\rangle$.

As $(P,I)$ is a standard pair, we have that $\overline{(P',I)}=(P,I)$. If $B \in P$ is a block not intersecting $I$, then $B$ must be obtained by merging blocks of $P'$  using the process (1) of Lemma \ref{l:Pmerge}. Let $B'\in P'$ be one of the blocks with $B' \subseteq B$, and consider $l_{B'}$. We claim that we may choose $l_B=l_{B'}$.

As $B$ is a union of translates of $B'$, it follows that  that $G^{B'}=G^B$. Moreover, $l_{B'} \in G_{B'} \le G_B$. Suppose that $i \in I$ satisfies  $G_{B'}=\langle l_{B'}, G^{B'}\rangle=\langle l_{B}, G^B\rangle \le G_i$.  If  $k \in G_B$, then $k \cdot B' \subseteq B$. It follows that $B'$ and $k\cdot B'$ were merged when obtaining $B$ using process (1), and hence that $k \in G'_{B'}$.
But as $G_{B'} \le G_i$, we have that   $k \in G_i$, as required.

Let $(P,I)=(\{X\},\varnothing)$ so that $D_{P,I}=\varnothing$. If $|\{x\in X: G_x=G\}|\neq 1$, then for any $x \in X$, there exists an $x' \in X$, $x \neq x'$ with $G_x=G_{x'}$. Hence, we may construct an $h \in \End_G(X)$ that does not fix any elements of $X$. Taking $g$ as the identity, we obtain $D_{g,h}=\varnothing=D_{P,I}$.

Conversely, let $x'$ be the unique element in  $\{x\in X: G_x=G\}$. Then every $G$-endomorphism of $X$ fixes $x'$. It follows that for any $g,h \in \End_G(X)$, the constant endomorphism with image $x'$ will be in $D_{g,h}$, and so $D_{g,h} \neq D_{P,I}$.
\end{proof}

We call a standard pair $(P,I)$ \emph{valid} if it satisfies the conditions of Lemma \ref{l:spairs} that guarantee the existence of $g,h \in \End_G(X)$ with $D_{P,I}=D_{g,h}$. The following theorem summarizes our results.

\begin{theorem}\label{t:taudomains}
The sets of the form $D_{g,h}$, with $g,h \in \End_G(X)$ are in bijective correspondence with the valid standard pairs of the $G$-set $X$.
\end{theorem}

We will now examine the partial automorphisms of $\End_G(X)$ of the form $\phi_{g,h}$.

\begin{lemma}
Let $g,h\in \End_G(X)$, with  $D_{g,h}= D_{P,I}$ and $D_{h,g}=D_{P',I'}$. Then $I$ is accessible from $P$ if and only if $I'$ is accessible from $P'$.
\end{lemma}
\begin{proof}
As $\phi_{g,h}$ is an automorphism from $D_{g,h}$ to $D_{h,g}$,  we obtain that  $D_{P,I}= \varnothing$ if and only if $D_{P',I'}= \varnothing$. The result now follows with Theorem \ref{t:tauI-sets}.
\end{proof}

\begin{lemma}\label{l:taualpha}
Let $g,h\in \End_G(X)$, with  $D_{P,I}=D_{g,h} \neq \varnothing$ and $D_{P',I'}=D_{h,g} \neq \varnothing$.
Then $g|_{I'}: I'\to I$, $h|_{I}:I\to I'$ are inverse $G$-isomorphisms. In particular, $I$ and $I'$ contain the same number of $G$-orbits of $X$ for each point stabilizer.

Conversely, let $I$ and $I'$ be unions of $G$-orbits such that  they contain the same number of $G$-orbits of $X$ for each point stabilizer. Then there exist a $G$-isomorphism $h'$ from $I$ to $I'$. Moreover,  if $\{G_i: i \in I\}$ contains every maximal stabilizer of $\{G_x:x \in X\}$, then any such $h'$ and its inverse $g'$ can be extended to
 $g,h \in \End_G(X)$ such that   $D_{g,h}= D_{P,I}$ and $D_{h,g}=D_{P',I'}$ for some $P, P'$ that can access $I,I'$, respectively.
\end{lemma}
\begin{proof}
Assume the conditions of the first part of the lemma.
By the proof of Theorem  \ref{t:tauI-sets}, $I$ and $I'$ are obtained in $\End_G(X)$ as in $T(X)$. Hence by Lemma \ref{l:alpha} $g|_{I'}$ and $h|_I$ are inverse bijections (note that the roles of $g$ and $h$ are switched in this lemma, as we reversed the order of composition in this section). As $g$ and $h$ are $G$-endomorphisms, the corresponding restrictions are $G$-isomorphisms. The result about the $G$-orbits of $I$ and $I'$ now
 follows from the $G$-isomorphism $h|_I$.

For the converse part, it is clear the $h'$ exists.  Pick a set of representatives $\{x_j\}$ for the $G$-orbits of $X$ not intersecting $I$, and extend $h'$ to $h$ by assigning $h(x_j)$ to any element $i'$ in $I'$ with $G_{x_j} \le G_{i'}$. Dually, extent $g'$ to a $G$-endomorphism. Then $gh$ and $hg$ fix exactly the elements of $I$ and $I'$, respectively. The construction from Theorem \ref{t:tauI-sets} now shows that $D_{g,h}=D_{P,I}$,
where  $P$ is constructed as in the theorem. Due to our choices above, this $P$ has the additional property that every $B \in P$ intersects $I$ in exactly one point, implying that $I$ is accessible from $P$. The results about $D_{h,g}$ follows dually.
\end{proof}

We remark that if $D_{g,h}=D_{h,g} = \varnothing$, and  if we represent $D_{g,h}$ and $D_{h,g}$ by the standard pair $(\{X\}, \varnothing)$,  then this results holds trivially as well.

\begin{lemma}\label{l:taubeta}
Let $g,h\in \End_G(X)$, with  $D_{g,h}= D_{P,I}$ and $D_{h,g}=D_{P',I'}$. Assume that $(P,I)$ are $(P',I')$ are standard pairs.

Then $\hat{h}: P \to P'$, given by $\hat{h}([x]_P)  = [hx]_{P'}$, and $\hat{g}: P'\to P$, given by
$ \hat{g}([x']_{P'}) = [gx']_{P}$, are well-defined inverse bijections that are compatible with the action of $G$ on $P$ and $P'$.
If $\hat{h}B=B'$ for some $B \in P$, then $G^B=G^{B'}$ and $G_B=G_{B'}$. A dual statement holds for $\hat{g}$.

In addition, for all $B \in P$, $B' \in P'$, we get $B\cap I = \varnothing \iff\hat{h} B\cap I'= \varnothing$
and $B' \cap I' = \varnothing \iff \hat{g} B' \cap I= \varnothing$.

Conversely, suppose $P$ and $P'$ are non-null $G$-invariant partitions, $\{B_i\}$ and $\{B_i'\}$ are representatives of the orbits of $P$ and $P'$ of the same size, indexed such that $G_{B_i}=G_{B'_i}$ and $G^{B_i}=G_{B'_i}$ for all $i$. Then there exist  $g,h\in \End_G(X)$ such that $\hat{h}$ and $\hat{g}$ are inverse bijections between $P$ and $P'$ satisfying $\hat{h}B_i=B'_i$.
\end{lemma}

\begin{proof}
If $D_{g,h}=\varnothing$, or $|I|=|I'|=1$, then $P=P'=\{X\}$, and the first part of the statement holds trivially. So assume that this is not the case.

Then we may apply Lemma \ref{l:beta} to $g,h$, which shows that $h,g$ induce inverse bijections $\bar{h}$ and $\bar{g}$ between $\bar{P}$ and $\bar{P}'$, the partitions corresponding to the connected components of the function graphs of $h,g$ (we recall that the roles of $g,h$ are exchanged compared to the lemma, as we are using right-to-left composition in this section). Moreover, the lemma shows that $\hat{h}$  maps blocks  intersecting $I$ to blocks intersecting $I'$, and vice versa for $\hat{g}$.
 As $g,h \in \End_G(X)$, we also obtain that these bijections are compatible with the action of $G$, and moreover that $G_B=G_{\bar{h}B}$ and $G^B=G^{\bar{h}B}$ for all $B \in \bar{P}$ (and dually for $\bar{g}$).

We claim that $\bar{h}$, and $\bar{g}$ in turn induce inverse bijections  from  $P$ to $P'$, using hat $\overline{(\bar{P},I)}=(P,I)$ and $\overline{(\bar{P}',I)}=(P',I)$. Suppose first that $B, l \cdot B \in \bar{P}$ are such that $l \in G'_B$, as calculated in the pair $(\bar{P},I)$. Because $\bar{h}$ is compatible with $G$, and $h$ maps $I$ to $I'$ by Lemma \ref{l:taualpha}, it follows that $G'_B=G'_{\hat{h}B}$, where the latter is calculated according to the pair $(\bar{P}',I')$. Hence $\bar{h}B$ and $l \cdot \bar{h}B$ lie in the same block of $P$. A dual statement holds for the action of $\bar{g}$. Hence if $B$ and $l\cdot B$ are merged according to part (1)\ of Lemma \ref{l:Pmerge} if and only if the same holds for their images under $\bar{h}$. It is straightforward to check that all other claimed properties hold as well.

By a very similar argument, we can show that $B,B' \in \bar{P}$ satify the conditions of part (2) of Lemma \ref{l:Pmerge} (with respect to $(\bar{P},I)$) if and only if $\bar{h}B$ and $\bar{h}B'$ do so  (with respect to $(\bar{P'},I')$). It follows that $\bar{h}$ and $\bar{g}$ induce bijections $\hat{h}$ and $\hat{g}$, as required.

For the converse part, let the $B_i, B'_i$ be as stated. For each $i$, choose representatives $x_{i,j}$ of the $G$-orbits intersecting $B_i$. As mentioned before, the union of the $x_{i,j}$ is a set of representatives for all $G$-orbits on $X$. For each $B'_i$ pick an $x'_i \in G^{B'_i}$, set $h(x_{i,j})=x'_i$, and extend to a $G$-endomorphism. We have that $h$ is well-defined, as $G_{x_{i,j}}\le G^{B_i}=G^{B'_i}=G_{x'_i}$. Moreover, as $G_{B_i}=G_{B'_i}$ the induced function $\hat{h}$ on $P$ is injective. Defining $g$ dually, we obtain $\hat{h}$ and $\hat{g}$ with the claimed properties.
\end{proof}

We now can derive a classification theorem for the generating elements  $\phi_{g,h}$ of the partial inner automorphism monoid.

\begin{theorem}\label{t:taugenerators}
The partial inner automorphisms of $\End_G(X)$ having the form $\phi_{g,h}$ are in bijective correspondence with the
tuples $(P,  P',I, I', \alpha, \beta)$, where
\begin{itemize}
\item $(P,I)$ and $(P',I')$ are valid standard pairs;
\item $\alpha$ is a bijection from $I$ to $I'$ that is compatible with the action of $G$;
\item $\beta$ is a bijection from $P$ to $P'$ that is compatible with the action of $G$, extends the function induced by $\alpha$ on $P$, and satisfies
$G^{\beta(B)}=G^B$
\end{itemize}
such that
\begin{itemize}
\item the domain of $\phi_{g,h}$ consists of all endomorphisms $t \in \End_G(X)$ with $\ima t \subseteq I$, $P \subseteq \ker t$;
\item the image of $\phi_{g,h}$ consists of all endomorphisms $t \in \End_G(X)$ with $\ima t \subseteq I'$, $P' \subseteq \ker t$;
\item given $t$ in the domain of $\phi_{g,h}$, and $x \in X$, we have
$\phi_{g,h}(t)(x)= \alpha i$, where $i\in I$ is the unique element in $t(\beta^{-1}([x]_{P'}))$.
\end{itemize}
For given valid standard pairs  $(P,I)$ and $(P',I')$, $\alpha$ and $\beta$ as above exist if and only if:
\begin{enumerate}
\item for each subgroup $H \le G$, the number of orbits in $I$ and $I'$ with point stabilizer $H$ are equal;
\item associating to each $G$-orbit $[B]$ of $P$ and $P'$ the pair $(G_B, G^B)$, the number of orbits corresponding to each such pair in $P$ and $P'$ are equal.
\end{enumerate}
\end{theorem}
\begin{proof}
Given $g,h \in\End_G(X)$, let $(P,I)$, $(P',I')$ be the  valid standard pairs with $D_{g,h}=D_{P,I}$, $D_{h,g}=D_{P',I'}$ (uniquely determined by Theorem \ref{t:taudomains}), and define $\alpha=h|_I$ and $\beta= \hat{h}$, as constructed in Lemma \ref{l:taubeta}.

That $\alpha$ and $\beta$ are bijections with the stated compatibility properties follows from Lemmas \ref{l:taualpha} and \ref{l:taubeta}.  Moreover if $i \in I$, then $\beta([i]_P)= \hat{h}([i]_p)=[h(i)]_{P'}=[\alpha(i)]_{P'}$, and so $\beta$ extends the function induced by $\alpha$.

The domain and image of $\phi_{g,h}$ have the stated form by Theorem \ref{t:taudomains}. Finally if $t \in \dom \phi_{g,h}$, and $x \in X$, let $i \in I$ be the unique element in $t(\beta^{-1}([x]_{P'}))$, then
\[
\phi_{g,h}(t)(x)=htg(x) \in htg([x]_{P'})= ht\hat{g}([x]_{P'})=ht\hat{h}^{-1}([x]_{P'})=ht\beta^{-1}([x]_{P'})=h(\{i\})=\{\alpha(i)\}\,,
\]
where we used that $\hat{g}=\hat{h}^{-1}$ by Lemma \ref{l:taubeta}.

The map $\phi_{g,h}$ uniquely determines the standard pairs $(P,I)$ and $(P',I')$  by Theorem \ref{t:taudomains}. We claim that this also holds for $\alpha$ and $\beta$.  For the uniqueness of $\alpha$, it suffices to show that for each $i \in I$, there exists $t_i \in D_{P,I}$ with $i \in \ima t_i$ (noting that $\ima \beta^{-1}=P$).
We may  construct such a $t_i \in D_{P,I}$ with $t_i(i)=i$ as in the last paragraph of the proof of Theorem \ref{t:tauI-sets}.

Suppose that $\beta, \beta':P\to P'$ are two bijections, that, together with some $\phi_{g,h},\alpha, P,I,P',I'$ satisfy the conditions of the theorem. We may assume that $\phi_{g,h}$ is not the empty mapping.

Let $C \in P'$, $B= \beta^{-1}(C), B'=\beta'^{-1}(C)$. For any $t \in D_{P,I}$, we obtain that
\[
\alpha(t(B))=\alpha(t(\beta^{-1}(C)))=\phi_{g,h}(C)=\alpha(t(\beta'^{-1}(C)))=\alpha(t(B'))\,.
\]
As $\alpha$ is bijective, and $t(B), t(B')$ are singletons, we obtain that $t(B)=t(B')$ for all $t \in D_{P,I}$.

Assume, by way of contradiction, that $B$ and $B'$ do not lie in the same $G$-orbit on $P$. Pick representative sets $B_i$, and representative elements $\{x_n\} \subseteq \bigcup B_i$, as before, with $B =B_1, B'=B_2$.
As $(P,I) \neq (\{X\}, \varnothing)$ is a standard pair, $I$ is accessible from $P$, and hence for each $B_i$ there exists $j_i \in I$ with $G_{B_i} \le G_{j_i}$. Hence we may set $t(x_j) =j_i $ for all $x_j \in B_i$ and extent to a $G$-endomorphism $t \in D_{P,I}$.
If for either $j_1$ or $j_2$ there is more than one possible choice, we may obtain a $t$ with $t(B) \neq t(B')$ by chosing $j_1 \neq j_2$. Therefore, by our assumption, there is only one choice for $j_1$ and $ j_2$. However this implies that  $j_1=j_2$ is a sink, and $G_B, G_{B'} \not\le G_j$ for $j \neq j_1$. As $(P,I)$ is a standard pair, it follows that $B=B'$, for a contradiction.

Hence  $B'= l\cdot B$ for some $l \in G$. Once again because $(P,I)\neq (\{X\},\varnothing)$ is a standard pair, there exists $i \in I$ such that $G_B \le G_i$.
Clearly, there also exsist a $t\in D_{P,I}$ with $t(B)=\{i\}$ (for example, we can obtain $t$ by a straigtforward adoptation of the construction from the previous paragraph). Hence $\{i\}=t(B)=t(B')=l \cdot t(B)=\{l \cdot i\}$, and so $l \in G_i$. As this  holds for all $i\in I$ with $G_B \le G_i$, it follows that $l \in G'_B$. As $(P,I)$ is a standard pair, we obtain $B=B'$. Hence $\beta^{-1}, \beta'^{-1}$ agree on each $C \in P'$, and so $\beta=\beta'$, as required. It follows that the correspondence between $\phi_{g,h}$ and the given tuples is injective.

Before addressing the surjectivity of this correspondence, we first prove the final statement about the existence of $\alpha $ and $\beta$. Assume that $(P,I)$ and $(P',I')$ are valid standard pairs. If $I$ and $I'$ do not have the same number of orbits for every point stabilizer, then clearly $\alpha$ does not exist.  The same holds $P$ and $P'$ do not satisfy the condition (2), where we note that $G_B=G_{\beta (B)}$ is necessary for $\beta $ to be invertible.

Assume instead that $P,P',I,I'$ satisfy conditions (1) and (2). By Lemma \ref{l:taualpha}, there exists a $G$-isomorphism $\alpha: I \to I'$. Because of (2), we may pick orbit representatives $B_i, B'_i$ of $P,P'$ such that $G_{B_i}=G_{B'_i}$ and $G^{B_i}=G^{B'_i}$. Moreover, we note that for the standard pair $(P,I)$, if $j\in I$ then $G_{[j]_P}=G^{[j]_P}=G_j$, and that a corresponding statement holds for $(P',I')$. Thus we may assume that if $B_i=[j]_P$ then $B'_i=[\alpha j]_{P'}$. Now the second part of Lemma \ref{l:taubeta} shows the existence of $\beta$ with the required properties. Hence $\alpha , \beta$ exist if and only if (1) and (2) hold.

In fact, Lemma \ref{l:taubeta} shows that $\beta$ has the form $\hat{h}$ for some $h \in \End_G(X)$, and that $\beta^{-1}=\hat{g} $ for some $g\in \End_G(X)$, where $\hat{h}([x]_P)=[h(x)]_{P'}$ and  $\hat{g}([x]_{P'})=[g(x)]_{P}$.
By direct calculation, we obtain that $\phi_{g,h}$ and $(P,I,P',I',\alpha, \beta)$ satisfy the conditions of the theorem, and therefore our correspondence is surjective. The result follows.
\end{proof}

We will now turn our attention to general elements of $\Inn(\End_G(X))$. For consistency with our other conventions in the section, we will also redefine the operation of $\Inn(\End_G(X))$ as right-to-left partial composition. We accordingly also adapt the following definition from the previous section.

\begin{defi}\label{d:bar}
Let $P,P'$ be partitions of $X$, and $\gamma: P \to P'$ a bijection. If $\bar P= \{B_i\}$ is a partition that refines to $P$, we define
$\bar{\gamma}$ on $\bar P$ by $ \bar{\gamma} (\cup B_i)= \cup (\gamma (B_i) )$.
\end{defi}

Clearly, $\bar{\gamma}$ is well-defined, and its image is a partition that refines to $P'$. Moreover, if $\gamma $ is compatible with $G$ and $P,P', \bar P$ are $G$-invariant, then the same properties hold for $\bar{\gamma}$ and $\bar{\gamma}(\bar{P})$.  The next theorem describes the general elements of $\Inn(\End_G(X))$.

\begin{theorem}\label{t:taugeneral}
Let $\phi \in \Inn(\End_G(X))$. Then there exist
\begin{itemize}
\item $G$-invariant partitions $P, P'$ of $X$;
\item $I, I' \subseteq X$, each of which are unions of $G$-orbits;
\item $G$-compatibe bijections $\alpha: I \to I'$, $\beta: P \to P'$ satisfying $\beta ([i]_P) = [\alpha i ]_{P'}$ for all $i \in I$
\end{itemize}
such that
\begin{itemize}
\item the domain of $\phi$ consists of all endomorphisms $t \in \End_G(X)$ with $\ima t \subseteq I$, $P \subseteq \ker t$;
\item the image of $\phi$ consists of all endomorphisms $t \in \End_G(X)$ with $\ima t \subseteq I'$, $P' \subseteq \ker t$;
\item given $t$ in the domain of $\phi$, and $x \in X$, we have
$\phi(t)(x)= \alpha i $, where $i\in I$ is the unique element in $t(\beta^{-1}([x]_{P'}))$.
\end{itemize}
Moreover, if $\phi_1,\phi_2 \in \Inn(\End_G(X))$ have corresponding parameters
\[
(P_1,I_1, P'_1, I_1', \alpha_1, \beta_1)\mbox{ and }(P_2, I_2, P'_2, I_2', \alpha_2, \beta_2)
\]
then $\phi_2\phi_1$ corresponds to
\[
(\bar \beta_1^{-1}(P'_1 \vee P_2) ,\alpha_1^{-1} (I'_1 \cap I_2) ,\bar \beta_2 (P'_1 \vee P_2) ,  \alpha_2(I'_1 \cap I_2), \alpha_2\alpha_1, \bar \beta_2\bar \beta_1)\,,
\]
where $\alpha_2\alpha_1$ refers to the partial composition $\alpha_2\alpha_1|_{\alpha_1^{-1}(I'_1 \cap I_2) }$, and $\bar\beta_1^{-1}, \bar\beta_2$ are defined with respect to the partition $\bar{P}=P'_1 \vee P_2$.
\end{theorem}
\begin{proof}
Since $\Inn(\End_G(X))$ is defined in terms of generators, we will prove the assertions by inducting over the involved elements $\phi, \phi_1, \phi_2$, starting with those $\phi$ of the form $\phi_{g,h}$. This base case
follows from Theorem \ref{t:taugenerators}.

Now suppose the theorem holds for $\phi_1, \phi_2\in \Inn(\End_G(X))$. Then $L:=\ima \phi_1\cap \dom \phi_2$ consists of all endomorphisms $t$ with $\ima t\subseteq I'_1\cap I_2$ and $P_1'\vee P_2 \subseteq \ker t$. It is now straightforward to check that
\[
\phi_1^{-1} L =D_{ \bar \beta_1^{-1}(P'_1 \vee P_2), \alpha_1^{-1}(I'_1 \cap I_2) } \mbox{  and }\phi_2 L=D_{  \bar \beta_2(P'_1 \vee P_2),  \alpha_2(I'_1 \cap I_2)}
\]
and hence these parameters define the domain and image of $\phi_2\phi_1$, where we note that
\begin{itemize}
\item $\vee$, $\bar \beta_1^{-1}$, and $\bar \beta_2$ preserve $G$-invariance (the latter two because of $G$-compatibility of $\beta_2$ and $\beta_1$), and thus $\bar \beta_1^{-1}(P'_1 \vee P_2)$ and $\bar \beta_2(P'_1 \vee P_2)$ are $G$-invariant;
\item $\cap$, $\alpha_1^{-1}, \alpha_2$ preserve the property of being unions of $G$-orbits, hence  $ \alpha_1^{-1}(I'_1 \cap I_2)$ and $ \alpha_2(I'_1 \cap I_2)$ are such unions.
\end{itemize}
Moreover, $\alpha_2\alpha_1$ and $\bar \beta_2 \bar \beta_1$ are clearly $G$-compatible. The remaining proof is identical to the proof of Theorem \ref{t:general}, except for adjustments related to the inverted order of composition.
\end{proof}

We can now describe the structure of $\Inn(\End_G(X))$. For a finite Abelian $G$-set $X$, let $A(X)$ consists of all $G$-isomorphisms from $I$ to $I'$ where $I,I' \subseteq X$ are unions of $G$-orbits. In addition, let  $B(X)$ be the set of all $G$-compatible bijections from $P$ to $P'$, where $P, P'$ are $G$-invariant partitions of $X$.
We say that
$\alpha\in A(X), \alpha:I\to I'$ and $\beta\in B(X), \beta: P \to P'$ are compatible, written $\alpha\approx \beta$, if $\beta([i]_P) = [\alpha i]_{P'}$ for all $i \in I$.

Let $V(X) = \{(\alpha, \beta):\alpha\in A(X), \beta\in B(X), \alpha \approx \beta\}$.
For $z=(\alpha, \beta) \in V(X)$ we set $\dom z= (\dom \beta, \dom \alpha)$.
On $V(X)$ we define a binary operation by
\[
(\alpha_2, \beta_2)(\alpha_1, \beta _1)=(\alpha_2 \alpha_1, \bar \beta_2 \bar \beta_1)\,,
\]
where $\bar \beta_i$ is as in Theorem \ref{t:taugeneral}, and  where we define the domain of  $\alpha_2 \alpha_1$ [of $ \bar \beta_2 \bar \beta_1$] as the largest subset of $X$ [finest partition on $X$] for which these expressions are well-defined. It is easy to check that domains and images of  $\alpha_2 \alpha_1$ and $ \bar \beta_2 \bar \beta_1$ are given as in Theorem \ref{t:taugeneral}. Moreover, by essentially the same argument as in this theorem, we have that  $\alpha_2 \alpha_1 \in A(X), \bar \beta_2 \bar \beta_1 \in B(X)$, and $\alpha_2 \alpha_1 \approx \bar \beta_2 \bar \beta_1$, hence $(\alpha_2, \beta_2)(\alpha_1, \beta _1)\in V(X)$.

We now define a normal form on $V(X)$. We first extend the operator $\overline{(P,I)}$ to arbitrary $G$-invariant partitions $P$ and unions of $G$-orbits $I$, noticing that its definition only uses properties of the stabilizers $G_B$ and $G_i$ for $B \in P$ and $i \in I$, and hence is well-defined for general $(P,I)$. Concretely,
if $I$ is not accessible from $P$, then $\overline{(P,I)}=(\{X\}, \varnothing)$, otherwise, $\overline{(P,I)}=(P/\sim_P, I)$, where $\sim_P$ is defined as in Lemma \ref{l:Pmerge}.

For $(\alpha, \beta) \in V(X)$ with corresponding parameters $P,I,P',I'$, let $\overline{(\alpha, \beta)}:\overline{(P,I)} \to \overline{(P',I')}$ be as follows:
\begin{itemize}
\item if  $\dom \alpha$ is not accessible from $\dom \beta$, let $\overline{(\alpha, \beta)}=(\varnothing, \id_{\{X\}})$.
\item otherwise, let $\overline{(\alpha, \beta)}=(\alpha, \beta')$, where $\beta'$ is the mapping induced on  $P/\sim_P$ by $\beta$. 
\end{itemize}

\begin{lemma}
The mapping $\overline{(\alpha, \beta)}$ is well defined.
\end{lemma}
\begin{proof}
Let $(\alpha, \beta) \in V(X)$ have corresponding parameters $P,I,P',I'$.
Suppose first that $\dom \alpha$ is not accessible from $\dom \beta$. Then there exists $B \in P$, such that $G_B \not\le G_i$ for all $i \in I$. As $\alpha$ is a $G$-isomorphism, $\{G_i:i \in I\}=\{G_{\alpha i}:i \in I\}$, and as $\beta $ is $G$-compatible, $G_B=G_{\beta B}$. Hence, $\beta B$ witnesses that $\ima \alpha$ is not accessible from $\ima \beta$, and so $\overline{(P',I')}= (\{X\},\varnothing)$ and $\overline{(\alpha, \beta)}=(\varnothing, \id_{\{X\}})$ is well-defined.

Suppose instead that $\dom \alpha$ is accessible from $\dom \beta$. If $(Q,I)=\overline{(P,I)}$, then $Q$ is obtained from $P$ by merging blocks according to one of the processes described in Lemma \ref{l:Pmerge}. Suppose first that $B'=l \cdot B$, where $B \in P$, and $l \in G'_B$. Once again, we have that $\{G_i:i \in I\}=\{G_{\alpha i}:i \in I\}$, $G_B=G_{\beta B}$, and in addition, that $\beta B'=l \cdot \beta B$. It follows that $l \in G'_B=G'_{\beta B}$, and hence we obtain that $\beta B$ and $\beta B'$ are contained in the same block of $Q'$, where $\overline{(P',I')}=(Q',I)$.

By a similar argument, we show that if $B,B'\in P$ are merged by the second construction from Lemma  \ref{l:Pmerge}, then so are $\beta B$ and $\beta B'$, where we notice that if $i\in I$ is a sink, then so is $\alpha(i) \in I'$. It follows that $\beta(Q)$ is a refinement of $Q'$. Applying the same argument to $\alpha^{-1}$ and $\beta^{-1}$, we see that $\beta^{-1}(Q')$ is a refinement of $Q$, and so $\beta Q=Q'$. Thus $\beta$ induces a well defined function $\beta'$ on $Q$.

Clearly,  $\beta'$  is $G$-compatible, and  $\alpha \approx \beta'$ holds. The results follows.
\end{proof}

It immediately follows from the definition that $\overline{\overline{(\alpha, \beta)}}= \overline{(\alpha, \beta)}$.
On $V(X)$, let $\theta$ be the relation defined by $(\alpha, \beta)\theta (\alpha', \beta')$ if and only if $\overline{(\alpha, \beta)}= \overline{(\alpha', \beta')}$. Clearly, $\theta$ is an equivalence relation.

\begin{lemma}\label{l:tau-well}
The relation $\theta$ is compatible with the operation of $V(X)$.
\end{lemma}
\begin{proof}
We first note that if $I$ is not accessible from $P$, then the same will holds for any subset of $I$ and coarsening of $P$. It follows that if $(\alpha, \beta) \in V(X)$ has $\overline{(\alpha,\beta)}=(\varnothing, \id_{\{X\}})$, then so will $(\alpha, \beta)z$ and $z(\alpha, \beta)$, for any $z \in V(X)$.

Hence consider $z_1=(\alpha_1,\beta_1)$, $z_2=(\alpha_2,\beta_2) \in V(X)$ with $\overline{(\alpha_1,\beta_1)}\neq (\varnothing, \id_{\{X\}})\neq \overline{(\alpha_2,\beta_2)}$. It suffices to show that $z_2z_1 \theta \bar z_2 \bar z_1$, in other words that $\overline{z_2z_1}=\overline{\bar z_2 \bar z_1}$. This equation clearly holds if $ \overline{z_2z_1}=(\varnothing, \id_{\{X\}})$, and hence we may assume that does not occur as well.

Suppose that $\beta_1: P_1 \to P'_1$, $\alpha_1: I_1 \to I'_1$, $\beta_2:P_2 \to P'_2$ and $\alpha_2:I_2 \to I'_2$.
Then the domain of the $\beta$-component of $z_2z_1$  is
\[
K:=\bar \beta_1^{-1}(P'_1 \vee P_2)\,,
\]
where $\bar \beta_1^{-1}$ refers to the construct  from Definition \ref{d:bar}. Similarly, the domain of the $\beta$-component of  $\bar z_2 \bar z_1$ is given by
\[
L:=\bar \beta_1^{-1}(\bar P'_1 \vee \bar P_2)\,,
\]
with $(\bar P'_1, I'_1)= \overline{(P'_1, I'_1)}$ and $(\bar P_2, I_2)= \overline{(P_2, I_2)}$.
Suppose that $B,C \in P'_1$ are  such that $B,C$ are contained in the same block of $\bar P'_1$. Then either $C= l \cdot B$ where $l \in G'_B$ or there is a sink $i' \in I'$ such that $G_B, G_C \not\le G_i$ for $i \neq i'$.

Consider first the case where $C= l \cdot B$. Let $[B], [C]$ be the blocks containing $B,C$ in $P'_1 \vee P_2$. Then $G_B \le G_{[B]}$. Moreover, suppose we let
\[
G'_{[B]}=\bigcap_{i \in I'_1 \cap I_2: G_{[B]}\le G_i}G_i\,,
\]
then $G'_{B}\le G'_{[B]}$. Hence $l \in G'_{[B]}$. Let $B'=\beta_1^{-1}([B])$, $C'=\beta_1^{-1}([C])$. We obtain that in $L$, $C'= l \cdot B'$ with $l \in G'_{B'}$, where $G'_{B'}$ is calculated with respect to the partition $L$, and the set  $ \alpha^{-1}_1(I'_1 \cap I_2)$. Hence,  if
\[
(L',  \alpha^{-1}_1(I'_1 \cap I_2))=\overline{(L,  \alpha^{-1}_1(I'_1 \cap I_2))}\,,
\]
then $B'$ and $C'$,  lie in the same block of $L'$, which is the domain of the $\beta$-component of $\overline{z_2z_1}$.

Assume instead that there is a sink $i'\in I'_1$ such that $G_B, G_C \not\le G_i$ for $i \neq i'$.
 If $i'\in I_2$, then a similar argument shows that $B',C'$ lie in the same block of $L'$. If instead $i' \notin I_2$, then $B'$ cannot access $\alpha^{-1}_1(I'_1 \cap I_2)$. In this case $\overline{z_2z_1}= (\varnothing, \id_{\{X\}})$, which we already excluded above.

It follows that if $B,C \in P'_1$ lie in the same block of $\bar P'_1$, then $\bar \beta_1^{-1} [B]$ and $\bar \beta_1^{-1} [C]$  lie in the same block of $L'$, the domain of the $\beta$-component of $\overline{z_2z_1}$. A corresponding results holds for blocks $B,C \in P_2$ that are merged in $\bar P_2$. Hence, $\beta_1^{-1}(\bar P'_1 \vee \bar P_2 )$ is a refinement of $L'$, and therefore
\[
\dom \overline{\bar z_1 \bar z_2}=\overline{(\beta_1^{-1}(\bar P'_1 \vee \bar P_2 ), \alpha_1^{-1}(I'_1 \cap I_2))}
\]
is a refinement in the first coordinate (and equal in the second) of
\[
\overline{(L, \alpha_1^{-1}(I'_1 \cap I_2))}=(L', \alpha_1^{-1}(I'_1 \cap I_2))= \dom\overline{z_2z_1}\,.
\]
Clearly,  in the first coordinate $\dom\overline{z_2z_1}$ is a refinement of $\dom \overline{\bar z_1 \bar z_2}$. Hence these two domains are equal. By a dual argument, we obtain that their images are equal as well. As the components of  $\overline{z_2z_1}$ and $\overline{\bar z_1 \bar z_2}$ are both induced by $\alpha_2\alpha_1$ and $\beta_2\beta_1$ on the same domains, we obtain  $\overline{z_2z_1} =\overline{\bar z_1 \bar z_2}$, as required.
\end{proof}

Set $W(X)= V(X)/ \theta.$ For $[(\alpha, \beta)]_\theta \in W(X)$ we will use the short  notation $[\alpha, \beta]$. We are now able to describe $\Inn(\End_G(X))$ as a substructure of $W(X)$.

\begin{theorem}\label{t:embed2}
Let $X$ be any finite abelian $G$-set. For $\phi \in \Inn(\End_G(X))$, let $\alpha_\phi: I \to I', \beta_\phi:P \to P'$ be the  bijections associated with $\phi$ by \textup{Theorem \ref{t:taugeneral}}. Then $\varphi: \Inn(\End_G(X))\to W(X)$, given by $\varphi(\phi)=[\alpha_\phi, \beta_\phi]$ is an embedding.

In particular, $\Inn(\End_G(X))$ is isomorphic to the substructure of $W(X)$ generated by all elements of $W(X)$ that can be represented as $[\alpha, \beta]$ such that there exist $P,I,P',I',\alpha,\beta$ satisfying the following:
\begin{itemize}
\item $(P,I)$ and $(P',I')$ are valid standard pairs;
\item $\alpha$ is a bijection from $I$ to $I'$ that is compatible with the action of $G$;
\item $\beta$ is a bijection from $P$ to $P'$ that is compatible with the action of $G$, extends the function induced by $\alpha$ on $P$, and satisfies $G^{\beta(B)}=G^B$.
\end{itemize}
\end{theorem}
\begin{proof}
Our construction guarantees that $\varphi$ is a homomorphism, provided it is well defined.

Hence let $\phi \in \Inn(\End_G(X))$, and $P,I,P',I',\alpha, \beta$ be the parameters associated with $\phi$ by Theorem \ref{t:taugeneral}. It suffices to show that
$\phi$ can also be described by the pair $\overline{(\alpha, \beta)}$. This is clear if $\phi$ is the empty mapping. Otherwise, let $\overline{(P,I)}=(Q,I)$, $\overline{(P',I')}=(Q',I')$, and  $\overline{(\alpha, \beta)}= (\alpha, \beta')$. 

We can show that $D_{P,I}=D_{Q,I}$ and $D_{P',I'}=D_{Q',I'}$ exactly as in the proof of Lemma \ref{l:Pmerge}, as conditions (1)--(4)\ of Lemma \ref{t:tauI-sets} were not needed for proving these equalities.  Hence, let $t \in D_{P,I}=D_{Q,I}$, then for any $x \in X$, $t(\beta^{-1}([x]_{Q})=t(\beta'^{-1}([x]_{Q'})$, as $\beta'$ is the function induced by $\beta$ on $Q$. Thus $(\alpha, \beta')=\overline{(\alpha, \beta)}$ also describes $\phi$. It follows that $\varphi$ is a well-defined homomorphism.

That $\varphi$ is injective follows immediately from Theorem \ref{t:taugeneral}, while the final assertion follows from the description of the generators $\phi_{g,h}$ of $\Inn(\End_G(X))$ in Theorem \ref{t:taugenerators}.
\end{proof}

\section{Problems}
\label{Sec:questions}

The referee suggested the following interesting problem. 
Having a notion of inner automorphism leads rather naturally to wondering about outer automorphisms, particularly in light of Remark \ref{Rem:inverse} where it was noted that $\Inn(S)\cong S$ when $S$ is the symmetric inverse monoid. 

\begin{que}
	Is there a natural notion of \emph{outer} automorphism of a semigroup, at least for the inverse case? If so, what is the outer automorphism group of the symmetric inverse monoid? The symmetric group $S_6$ on $6$ symbols has an exceptional outer automorphism; does this happen for the inverse symmetric monoid on $6$ symbols?
\end{que}

Let $G$ be a group of permutations on a finite set $X$, and let $t\in T(X)\setminus G$ be a transformation. Denote by $\langle G,t\rangle$ the semigroup generated by $G$ and $t$. When $G$ has special properties (as defined in the classes of groups and classification theorems in
\cite{AC16,AABCS21,ABC21,ABC19,ACS17}), the semigroups  $\langle G,t\rangle$, for all non-permutations $t\in T(X)$, have a rich structure and deep interconnections with permutation groups, automata theory, combinatorics, and geometry. Thus it seems especially instructive to solve the following problem.

\begin{que}
	With $X$, $G$ and $t\in T(X)\setminus G$ as above, assume $G$ belong to one of the classes of permutation groups discussed in the papers mentioned above. Characterize the partial inner automorphism monoid of  $\langle G,t\rangle$.
\end{que}

Baumslag \cite{baumslag} has shown that the automorphism group  of a finitely generated residually finite group is itself residually finite. This result does not extend to the partial automorphism inverse monoid of a finitely generated residually finite semigroup. For example, if $B=\{b_i: i \in \mathbb{N}\} \cup \{0\}$ is an infinite null semigroup, $A= \langle a\rangle$ is infinite, and $S = A\cup B$ with $a^k b_l = b_l a^k = b_{k+l}$, then it is easy to see that the partial automorphisms of $S$ include a semigroup isomorphic to the symmetric group on an infinite set. The latter group is not residually finite because it contains the alternating group, which is an infinite simple group (see \cite{C-ST2023}, Lem.~2.6.3, p.~71). Thus there are partial automorphisms of $S$ which cannot be separated by a homomorphism into a finite semigroup.

However, we have been unable to resolve the following questions, suggested by the referee, for the inner partial automorphism monoid.

\begin{que}
	Let $S$ be a finitely generated residually finite semigroup. Is $\Inn(S)$ residually finite?
\end{que}

While Theorem \ref{Thm:RM} gives a general description of the inner automorphism monoid $\Inn(S)$ of a Rees matrix semigroup $S$, the referee pointed out further questions of interest.

\begin{que}
	Can two nonisomorphic Rees matrix semigroups have isomorphic partial inner automorphism monoids? What can be said about finite generation and finite presentability of $\Inn(S)$ for $S$ a Rees matrix semigroup?
\end{que}

In addition, we have left untouched the case of completly $0$-simple semigroups, that is, Rees $0$-matrix semigroups.

\begin{que}
	Characterize the partial inner automorphism monoids of Rees $0$-matrix semigroups.
\end{que}

We found the partial inner automorphisms of $T(X)$. A certainly much more demanding challenge is the following. 

\begin{que}
	Characterize the partial inner automorphism monoid of the partition monoid and other diagram monoids (such as Planar, Jones, Kauffman, Martin, Temperley/Lieb, etc.).
\end{que}

Regarding the endomorphism monoids of $G$-sets, we assumed that $G$ is abelian.
\begin{que}
	Extend the results of \S\ref{Sec:mappingG-set} on $G$-sets to the case of nonabelian groups $G$.
\end{que}

The following papers deal with several transformation semigroups (and some also deal with conjugacy):  \cite{ArKiKo_Inv,ArKiKoMaTA,AKM14,F02,GaMa10,Ko18,KuMa07,Steinberg15,St19}.
\begin{que}
	Extend the results on $T(X)$ in this paper to the transformation semigroups appearing in the papers cited above.
\end{que}

The theorems and problems in this paper have natural linear counterparts.
\begin{que} 
	Characterize the partial inner automorphism monoid of the endomorphism monoid of a finite dimensional vector space over a field.
\end{que}

When results hold for both finite sets and finite dimensional vector spaces, it is natural to try to extend the results to independence algebras \cite{G95,G07,CS00,ABCKK22}.

\begin{que}
	Characterize the partial inner automorphism monoid of the endomorphism monoid of a finite dimensional independence algebra. The classification of independence algebras \cite{ABCKK22} may be useful for this.
\end{que}

An even more general problem is the following. 
\begin{que}
	Characterize the partial inner automorphism monoid of the endomorphism monoid of free objects \cite{AMS03} or of the endomorphism monoid of algebras admitting some general notion of independence \cite{AW09}. Examples of the latter include the endomorphism monoids of $\mathit{MC}$-algebras, $M S$-algebras, $S C$-algebras, and $S C$-ranked algebras \cite[Ch.~8]{AW09}. The first step would should be to solve the problem for the endomorphism monoid of an $S C$-ranked free $M$-act \cite[Ch.~9]{AW09}, and for an $S C$-ranked free module over an $\aleph_1$-Noetherian ring \cite[Ch.~10]{AW09}.
\end{que}

Related to the previous problem, we have the following.
\begin{que}
	Since all varieties of bands are known, describe the partial inner automorphism monoid of the endomorphism monoid of the free objects in each variety of bands (for details and references, see \cite{AK07}).
\end{que}

\begin{que}
	Let $C$ be a class of inverse monoids with zero (say, chains, semilattices with identity, Clifford monoids, etc). Characterize the semigroups $S$ such that the partial inner automorphism monoid of $S$ belongs to $C$. 
\end{que}

\section*{Acknowledgements}
This work is funded by national funds through the FCT - Fundaç\~{a}o para a Ci\^{e}ncia e a Tecnologia, I.P., under the scope of the projects UIDB/00297/2020 and UIDP/00297/2020 (Center for Mathematics and Applications).




\begin{thebibliography}{99}

\bibitem{AABCS21}
J. Ara\'{u}jo, J.P. Ara\'{u}jo, W. Bentz, P.J. Cameron, and P. Spiga,
A transversal property for permutation groups motivated by partial transformations,
\emph{J. Algebra} \textbf{573} (2021), 741--759.

\bibitem{ABC19}
J. Ara\'{u}jo, W. Bentz, and P.J. Cameron,
Orbits of primitive k-homogenous groups on $(n-k)$-partitions with applications to semigroups,
\emph{Trans. Amer. Math. Soc.} \textbf{371} (2019), 105--136.

\bibitem{ABC21}
J. Ara\'{u}jo, W. Bentz, and P.J. Cameron,
The existential transversal property: a generalization of homogeneity and its impact on semigroups,
\emph{Trans. Amer. Math. Soc.} \textbf{374} (2021), 1155--1195.

\bibitem{ABCKK22}
J. Ara\'{u}jo, W. Bentz, P.J. Cameron, M. Kinyon, and J. Konieczny,
Matrix theory for independence algebras,
\emph{Linear Algebra Appl.} \textbf{642} (2022), 221--250.

\bibitem{ABKKMM}
J. Ara\'{u}jo, W. Bentz, M. Kinyon, J. Konieczny, A. Malheiro, and V. Mercier,
\textit{Conjugacy in abstract semigroups, transformation and diagram
monoids, and conjugacy growth},
to appear.

\bibitem{AC16}
J. Ara\'{u}jo and P.J. Cameron,
Two generalizations of homogeneity in groups with applications to regular semigroups,
\emph{Trans. Amer. Math. Soc.} \textbf{368} (2016), 1159--1188.

\bibitem{ACS17}
J. Ara\'{u}jo, P.J. Cameron, and B. Steinberg,
Between primitive and 2-transitive: synchronization and its friends,
\emph{EMS Surv. Math. Sci.} \textbf{4} (2017), 101--184.

\bibitem{ArKiKo_Inv}
J. Ara\'{u}jo, M. Kinyon, and J. Konieczny,
Conjugacy in inverse semigroups,
\emph{J. Algebra}, \textbf{533} (2019), 142--173.

\bibitem{ArKiKoMaTA}
J. Ara\'{u}jo, M. Kinyon, J. Konieczny, and A. Malheiro,
Four notions of conjugacy for abstract semigroups,
\emph{Proc. Roy. Soc. Edinburgh Sect. A}
\textbf{147} (2017), 1169--1214.

\bibitem{AK07}
J. Ara\'{u}jo and J. Konieczny,
Automorphisms of endomorphism monoids of relatively free bands,
\emph{Proc. Edinb. Math. Soc.} \textbf{50} (2007), 1--21.

\bibitem{AKM14}
J. Ara\'{u}jo, J. Konieczny, and A. Malheiro,
Conjugation in semigroups,
\emph{J. Algebra} \textbf{403} (2014), 93--134.

\bibitem{AMS03}
J. Ara\'{u}jo, J.M. Mitchell, and N. Silva,
On generating countable sets of endomorphisms,
\emph{Alg. Univers.} \textbf{50} (2003), 61--67.

\bibitem{AW09}
J. Ara\'{u}jo and F. Wehrung,
Embedding properties of endomorphism semigroups,
\emph{Fund. Math.} \textbf{202} (2009), 125--146.

\bibitem{baumslag}
G. Baumslag,
Automorphism groups of residually finite groups,
\emph{J. London Math. Soc.} \textbf{38} (1963), 117--118.

\bibitem{BorralhoKinyon2020}
M. Borralho and M. Kinyon,
Variants of epigroups and primary conjugacy,
\emph{Comm. Algebra} \textbf{48} (2020), no.~12, 5465--5473.

\bibitem{pjc}
P.J. Cameron,
Endomorphisms and partial isomorphisms,
\emph{Semigroup Forum} (2025). \url{https://doi.org/10.1007/s00233-025-10514-5}

\bibitem{CS00}
P.J. Cameron and C. Szabó,
Independence algebras,
\emph{J. London Math. Soc. (2)} \textbf{61} (2000), 321--334.

\bibitem{C-ST2023}
T. Ceccherini-Silberstein and M. Coornaert,
\emph{Cellular automata and groups}, 2nd ed.,
Springer Monogr. Math., Springer, Cham, 2023.

\bibitem{ClPr64}
A.H. Clifford and G.B. Preston,
\emph{The Algebraic Theory of Semigroups},
Mathematical Surveys, No.~7, Amer.\ Math.\ Soc., Providence, Rhode Island, 1964 (Vol.~I) and 1967 (Vol.~II).

\bibitem{Derech2}
V. D. Derech,
Classification of finite commutative semigroups for which the inverse monoid of local automorphisms is a $\Delta$-semigroup,
\emph{Ukraïn. Mat. Zh.} \textbf{67} (2015), 867--873; translation in \emph{Ukrainian Math. J.} \textbf{67} (2015), 981--988.\\

\bibitem{Derech}
V. D. Derech,
Complete classification of finite semigroups for which the inverse monoid of local automorphisms is a permutable semigroup,
\emph{Ukra\"{i}n. Mat. Zh.} \textbf{68} (2016), 1571--1578; translation in \emph{Ukrainian Math. J.} \textbf{68} (2017), 1820--1828

\bibitem{Derech1}
V. D. Derech,
Complete classification of finite semigroups for which the inverse monoid of local automorphisms is a $\Delta$-semigroup,
\emph{Semigroup Forum} \textbf{102} (2021), 397--407.

\bibitem{Domanov}
O. I. Domanov,
Semigroups of all partial automorphisms of universal algebras. (Russian)
\emph{Izv. Vys\v{s}. U\v{c}ebn. Zaved. Matematika} \textbf{1971} (1971), no. 8(111), 52--58.

\bibitem{F02}
V.H. Fernandes,
Presentations for some monoids of partial transformations on a finite chain: a survey,
\emph{Semigroups, Algorithms, Automata and Languages, Coimbra, 2001}, 363--378, World Sci. Publ., River Edge, NJ (2002).

\bibitem{FitzGerald1}
D. G. FitzGerald,
Representations of inverse monoids by partial automorphisms,
\emph{Semigroup Forum} \textbf{61} (2000), no.~3, 357--362.

\bibitem{GaMa10}
O. Ganyushkin and V. Mazorchuk,
\emph{Classical finite transformation semigroups:\ an introduction},
Algebra and Applications \textbf{9},
Springer-Verlag, London, 2010.

\bibitem{GarretaGray2021}
A. Garreta and R. D. Gray,
On equations and first-order theory of one-relator monoids,
\emph{Inform. Comput.} \textbf{281} (2021), Paper No.~104745, 19~pp.

\bibitem{Goberstein1}
S. M. Goberstein,
$\mathcal{PA}$-isomorphisms of inverse semigroups,
\emph{Algebra Universalis} \textbf{53} (2005), 407--432.

\bibitem{Goberstein2}
S. M. Goberstein,
Inverse semigroups determined by their partial automorphism monoids,
\emph{J. Aust. Math. Soc.} \textbf{81} (2006), 185--198.

\bibitem{G95}
V. Gould,
Independence algebras,
\emph{Alg. Univers.} \textbf{33} (1995), 294--318.

\bibitem{G07}
R. Gray,
Idempotent rank in endomorphism monoids of finite independence algebras,
\emph{Proc. Roy. Soc. Edinb. Sect. A} \textbf{137} (2007), 303--331.

\bibitem{Howie}
J.M. Howie,
\emph{Fundamentals of semigroup theory},
Oxford University Press, New York, 1995.

\bibitem{Jack2023}
T. Jack,
On the complexity of inverse semigroup conjugacy,
\emph{Semigroup Forum} \textbf{106} (2023), no.~3, 618--632.

\bibitem{Jack2024}
T. Jack,
\emph{Answering five open problems involving semigroup conjugacy},
arXiv:2411.16284, 2024.

\bibitem{KS19}
M. Kinyon and D. Stanovsk\'{y},
Abelianness and centrality in inverse semigroups,
\emph{Proc. Edinburgh Math. J.}, to appear. 
\url{arXiv:2209.13805}.

\bibitem{Ko18}
J. Konieczny,
A new definition of conjugacy for semigroups,
\emph{J. Algebra and Appl.} \textbf{17} (2018), 1850032, 20 pp.

\bibitem{KuMa07}
G. Kudryavtseva and V. Mazorchuk,
On conjugation in some transformation and Brauer-type semigroups,
\emph{Publ. Math. Debrecen} \textbf{70} (2007), 19--43.

\bibitem{LiuWang2022}
X. Liu and S. Wang,
The transitivity of primary conjugacy in regular $\omega$-semigroups,
\emph{Open Math.} \textbf{20} (2022), no.~1, 1479--1508.

\bibitem{Mitsch}
H. Mitsch,
A natural partial order for semigroups,
\emph{Proc. Amer. Math. Soc.} \textbf{97} (1986), 384--388.

\bibitem{NarendranOtto1985}
P. Narendran and F. Otto,
Complexity results on the conjugacy problem for monoids,
\emph{Theoret. Comput. Sci.} \textbf{35} (1985), no.~2--3, 227--243.

\bibitem{NarendranOttoWinklmann1984}
P. Narendran, F. Otto, and K. Winklmann,
The uniform conjugacy problem for finite Church-Rosser Thue systems is NP-complete,
\emph{Inform. Control} \textbf{63} (1984), no.~1--2, 58--66.

\bibitem{Otto1984}
F. Otto,
Conjugacy in monoids with a special Church-Rosser presentation is decidable,
\emph{Semigroup Forum} \textbf{29} (1984), no.~1--2, 223--240.

\bibitem{Preston}
G. B. Preston,
Inverse semigroups: some open questions,
\emph{Proceedings of a Symposium on Inverse Semigroups and their Generalisations (Northern Illinois Univ., DeKalb, Ill., 1973)}, pp. 122--139, Northern Illinois University, De Kalb, IL, 1973.

\bibitem{Preston2}
G. Preston,
A note on representations of inverse semigroups,
\textit{Proc. Amer. Math. Soc.} \textbf{8} (1957), 1144--1147.

\bibitem{Shevrin}
L.N. Shevrin,
Epigroups,
in \emph{Structural theory of automata, semigroups, and universal algebra}, 331--380, NATO Sci. Ser. II Math. Phys. Chem., 207,
Springer, Dordrecht, 2005

\bibitem{Steinberg15}
B. Steinberg,
\emph{The representation theory of finite monoids},
Springer Monographs in Mathematics, Springer, New York, 2015.

\bibitem{St19}
B. Steinberg,
Linear conjugacy,
\emph{Canad. Math. Bull.} \textbf{62} (2019), 886--895.

\bibitem{Zhang}
L. Zhang,
Conjugacy in special monoids,
\emph{J. Algebra} \textbf{143} (1991), no.~2, 487--497.

\bibitem{Zhang1992}
L. Zhang,
Applying rewriting methods to special monoids,
\emph{Math. Proc. Cambridge Philos. Soc.} \textbf{112} (1992), no.~3, 495--505.

\end{thebibliography}
\end{document}